\documentclass{amsart}
\usepackage{amssymb,mathtools}
\usepackage[british]{babel}
\usepackage{enumitem}
\usepackage[latin1]{inputenc}
\usepackage{xcolor}

\usepackage{hyperref}

\newtheorem{theorem}{Theorem}
\numberwithin{theorem}{section}
\newtheorem{corollary}[theorem]{Corollary}
\newtheorem{lemma}[theorem]{Lemma}
\newtheorem{proposition}[theorem]{Proposition}
\theoremstyle{definition}
\newtheorem{definition}[theorem]{Definition}
\newtheorem{remark}[theorem]{Remark}

\newtheorem{assumption}[theorem]{Standing Assumption}

\numberwithin{equation}{section}

\newcommand{\supp}{\operatorname{supp}}
\newcommand{\rng}{\operatorname{rng}}

\title[Partial impredicativity and well-ordering principles]{Fra\"iss\'e's conjecture, partial impredicativity and well-ordering principles, part II}
\author{Anton Freund}
\author{Katarzyna W. Kowalik}
\author{Davide Manca}

\address{Anton Freund, University of W\"urzburg, Institute of Mathematics, Emil-Fischer-Str.~40, 97074 W\"urz\-burg, Germany}
\email{anton.freund@uni-wuerzburg.de}

\address{Katarzyna W. Kowalik, Faculty of Philosophy, University of Warsaw, Krakowskie Przedmie\'scie 3, 00-047 Warsaw, Poland}
\email{kwz.kowalik@gmail.com}

\address{Davide Manca, Institute of Discrete Mathematics and Geometry, TU Wien, Wiedner Haupt\-strasse~\mbox{8-10}, 1040 Vienna, Austria}
\email{d.manca.math@gmail.com}

\thanks{The work of Freund has been funded by the Deutsche Forschungsgemeinschaft (DFG, German Research Foundation) -- Project number 460597863.}

\begin{document}

\begin{abstract}
We exhibit a well-ordering principle that is equivalent to a theory of partial impredicativity. The latter goes back to Towsner and relates to recent work of Suzuki and Yokoyama. Together with part~I of the present paper, this yields a combinatorial upper bound on the axiomatic strength of Fra\"iss\'e's conjecture.
\end{abstract}


\maketitle

\section{Introduction}

Reverse mathematics is a research program in mathematical logic, which aims to determine the minimal (set existence) axioms that are needed to prove a given mathematical theorem (see~\cite{simpson09}). It turns out that many theorems fall into one of five linearly ordered categories (`big five'), though a much more diverse landscape has emerged over the last years (see Figure~9.5 in~\cite{dzhafarov-mummert}). The two strongest categories are given by the axioms of arithmetical transfinite recursion ($\mathsf{ATR}_0$) and $\Pi^1_1$-comprehension ($\Pi^1_1\textsf{-CA}_0$). These are often classified as predicatively reducible and impredicative, respectively (see~\cite{feferman-predicativity} for philosophical background).

There are several mathematical theorems that are strictly weaker than $\Pi^1_1\textsf{-CA}_0$ (often for reasons of quantifier complexity) but not (known to be) provable in~$\mathsf{ATR}_0$. To analyse these theorems, Towsner has suggested some intermediate axioms of partial impredicativity~\cite{partial-impred}. In work of Suzuki and Yokoyama, these have been extended into a hierarchy that exhausts all $\Pi^1_2$-consequences of $\Pi^1_1\textsf{-CA}_0$ (see~\cite{suzuki-yokoyama}).

It is a long-standing open problem in reverse mathematics to determine the exact strength of Fra\"iss\'e's conjecture. The latter is actually a theorem of Laver~\cite{laver71}, which asserts that the embeddability order on countable (or more generally on $\sigma$-scattered) linear orders admits no infinitely descending chains and no infinite antichains (i.e., that it is a well-quasi-order). Shore~\cite{shore-comp-wos} has shown that Fra\"iss\'e's conjecture entails $\mathsf{ATR}_0$ (over $\mathsf{RCA}_0$ by a boot-strapping argument in~\cite{freund-manca-fraisse}), while Montalb\'an~\cite{montalban-fraisse} has established that it can be proved in~$\Pi^1_1\textsf{-CA}_0$. It is thus natural to study Fra\"iss\'e's conjecture through the lens of partial impredicativity. This is particularly interesting since Fra\"iss\'e's conjecture is equivalent to further relevant theorems (see the introduction of~\cite{montalban-fraisse}).

By inspecting the argument of Montalb\'an, one of the present authors has shown that Fra\"iss\'e's conjecture can be proved in a certain theory~$\Pi^1_1\textsf{-CA}^\Gamma_0$ (see~\cite{freund-fraisse-partial-impred}). This theory is somewhat stronger than Towsner's theories of partial impredicativity. It lies between the first two stages of the hierarchy of Suzuki and Yokoyama (who denote it by $\beta^1_0\mathsf{RFN}(1;\mathsf{ATR})$). Somewhat informally, $\Pi^1_1\textsf{-CA}^\Gamma_0$ extends $\mathsf{ACA}_0$ by the statement that any set lies in an $\omega$-model of parameter-free $\Pi^1_1$-comprehension. Here we work over $\mathsf{ACA}_0$ (which is stronger than $\mathsf{RCA}_0$) to have a smooth encoding of countable $\omega$-models. When $\omega$-models occur in the object language, it will always be understood that they are countable and coded (as in Section~VII.2 of~\cite{simpson09}).

\begin{definition}\label{def:Pi11-Gamma}
  For $X\subseteq\mathbb N$ and an $\omega$-model~$\mathcal M\ni X$, we write $\mathcal M\vDash\Pi^1_1\textsf{-CA}(X)$ if
  \begin{equation*}
      \{n\in\mathbb N:\mathcal M\vDash\varphi(n,X)\}\in\mathcal M
  \end{equation*}
  holds for any $\Pi^1_1$-formula $\varphi$ that has no set parameters other than~$X$. The axiom $\Pi^1_1\textsf{-CA}^\Gamma_0$ states that each $X\subseteq\mathbb N$ lies in an $\omega$-model~$\mathcal M$ with $\mathcal M\vDash\mathsf{ATR}_0+\Pi^1_1\textsf{-CA}(X)$. We also write $\Pi^1_1\textsf{-CA}^\Gamma_0$ for the theory that extends~$\mathsf{ACA}_0$ by this axiom.
\end{definition}

We have mentioned that Fra\"iss\'e's conjecture can be proved in $\Pi^1_1\textsf{-CA}^\Gamma_0$. The aim of the present paper is to make this bound more concrete in a combinatorial sense. We will rely on the notion of well-ordering principle in order to achieve this aim. Somewhat informally, a well-ordering principle (of type one) states that some computable transformation of linear orders preserves well-foundedness. As an example, when~$X$ is a linear order, we consider
\begin{equation*}
    \omega^X=\{\langle x_1,\ldots,x_n\rangle:x_1\geq\ldots\geq x_n\text{ in }X\}
\end{equation*}
with the lexicographic ordering (think of $\langle x_1,\ldots,x_n\rangle$ as a formal Cantor normal form $\omega^{x_1}+\ldots+\omega^{x_n}$). As shown by Girard~\cite{girard87} (and in a different way by Hirst~\cite{hirst94}), the statement that $\omega^X$ is well-founded for any well-order~$X$ is equivalent to arithmetical comprehension over $\mathsf{RCA}_0$. This can be used to determine the strength of mathematical theorems such as Higman's lemma and the Hilbert basis theorem (though the latter only requires the well-foundedness of $\omega^\omega$; see~\cite{simpson-higman}). The literature contains a great number of analogous equivalences between natural axioms and order transformations as well as applications~\cite{rathjen-afshari,freund-3-bqo,marcone-montalban-hausdorff,marcone-montalban,rathjen-atr,rathjen-weiermann-atr,thomson-rathjen-Pi-1-1}. In connection with the present paper, we mention a result of Rathjen and Valencia-Vizca\'ino~\cite{rathjen-model-bi}, which gives a well-ordering principle that is equivalent to the statement that any set is contained in an $\omega$-model of transfinite induction (also known as bar induction). This last statement is equivalent to a principle of partial impredicativity, the arithmetical leftmost path principle~$\mathsf{ALPP}$ of Towsner~\cite{partial-impred} (see~\cite{freund-fraisse-partial-impred,suzuki-tlpp-omega} for proofs of the equivalence). Apart from this somewhat indirect connection, the present paper seems to be the first that relates well-ordering principles and partial impredicativity.

Let us now state our main result, though its meaning will only be clear once the order~$\psi(\Gamma_{\Omega+X})$ has been defined. We will motivate the definition in the following paragraphs, but details are deferred to the next section.

\begin{theorem}\label{thm:main}
    The following are equivalent over~$\mathsf{ACA}_0$:
    \begin{enumerate}[label=(\roman*)]
        \item The principle $\Pi^1_1\text{-}\mathsf{CA}^\Gamma_0$ holds.
        \item The linear order~$\psi(\Gamma_{\Omega+X})$ is well-founded for every well-order~$X$.
    \end{enumerate}
\end{theorem}

Given the theorem, one can readily determine the proof-theoretic ordinal of the theory $\Pi^1_1\textsf{-CA}^\Gamma_0$. To describe this ordinal, we would need several constructions that are not otherwise relevant for the present paper. So we simply refer to the general connection between well-ordering principles and proof-theoretic ordinals~\cite{arai-strength-wop,bcdf-wo-princ}.

It is a classical result that $\mathsf{ATR}_0$ has proof-theoretic ordinal~$\Gamma_0$, which is the first fixed point of the so called Veblen hierarchy (see the next section for more details). More generally, we can consider the $\alpha$-th fixed point~$\Gamma_\alpha$, which can be represented by a term system that includes the elements of~$\alpha$ as constants. The symbol~$\Omega$ that occurs in the theorem above refers to an ordinal with strong closure properties. While the precise meaning will only become clear in the next section, it may help to think of the first uncountable ordinal. Then $\psi(\Gamma_{\Omega+X})$ should be interpreted as a countable ordinal that admits an order-preserving but partial function
\begin{equation*}
    \psi:\Gamma_{\Omega+X}\rightharpoonup\psi(\Gamma_{\Omega+X}).
\end{equation*}
 Note that there can be no total order embedding of $\Gamma_{\Omega+X}$ into $\psi(\Gamma_{\Omega+X})$, since the former ordinal is larger than the latter. By demanding that $\psi$ is defined on `most' arguments, we force $\psi(\Gamma_{\Omega+X})$ to be large. We will also write $\psi$ for an extension of the indicated function, which will be total but only order-preserving on `most' pairs of arguments. The function $\psi$ is a typical collapsing function as used for impredicative ordinal analysis (see Section~5.3 of~\cite{rathjen-sieg-stanford}).

To provide some intuition on the combinatorial side of our well-ordering principle, we recall Kruskal's theorem. The latter asserts that there are no infinite antichains (and of course no infinitely descending sequences) in the set of finite trees ordered by infimum-preserving embeddability (i.e., that we have a well-quasi-order). This remains true when the trees have leaf labels from a given well-quasi-order. By a classical result of H.~Friedman, Kruskal's theorem is unprovable in~$\mathsf{ATR}_0$ (see~\cite{simpson85}). This is a particularly important example of the incompleteness phenomenon from G\"odel's theorem, also because Kruskal's theorem is relevant for computer science. Its precise strength has been investigated from different viewpoints~\cite{rathjen-weiermann-kruskal,partial-impred}, and one may draw the informal conclusion that Kruskal's theorem is much closer to $\mathsf{ATR}_0$ than to $\Pi^1_1\textsf{-CA}_0$. Now Kruskal's theorem with labels is equivalent over~$\mathsf{RCA}_0$ to the assertion that $\vartheta(\Omega^\omega\cdot X)$ is well-founded for any well-order~$X$ (see~\cite{buriola-weiermann-kruskal}). Here $\vartheta$ is a collapsing function similar to~$\psi$. In the previously published first part~\cite{freund-fraisse-partial-impred} of the present paper, Theorem~\ref{thm:main} was actually announced with $\vartheta$ at the place of~$\psi$. We have decided to work with $\psi$ because it is more suitable for extensions of the result, which we envisage for future work. It is standard (though tedious) to translate between the results for $\vartheta$ and for~$\psi$ (cf.~\cite{buchholz-bachmann-howard,rathjen-weiermann-kruskal}). In any case, the analysis in terms of well-ordering principles allows for particularly informative comparisons between different theorems, e.g., between Kruskal's theorem and Fra\"iss\'e's conjecture. Let us also note that theories of related strength have been investigated in~\cite{buchholtz-jaeger-strahm}.

To conclude this introduction, we summarize the structure of our paper. As we have mentioned above, Section~\ref{sect:ordinal-notations} defines the order $\psi(\Gamma_{\Omega+X})$ and shows its central properties. In Section~\ref{sect:well-ordering-proof}, we establish the direction from~(i) to~(ii) of Theorem~\ref{thm:main}. The remaining sections are devoted to the proof of the converse: Section~\ref{sect:ind-def} builds on classical work to reformulate $\Pi^1_1\textsf{-CA}^\Gamma_0$ in terms of inductive definitions. In Section~\ref{sec:omega-completeness} we use Sch\"utte's proof of $\omega$-completeness (see \cite{schuette56} and the applications in~\cite{rathjen-afshari,jaeger-strahm-bi-reflection}) to construct the $\omega$-model that is required by $\Pi^1_1\textsf{-CA}^\Gamma_0$. This relies on the assumption that there is no $\omega$-proof of contradiction in the aforementioned theory of inductive definitions. To discharge this assumption, we adapt classical work from ordinal analysis. Specifically, our ordinal analysis has a predicative part (due to the occurrence of $\mathsf{ATR}_0$ in $\Pi^1_1\textsf{-CA}^\Gamma_0$), which is carried out in Section~\ref{sect:pred-ord-ana}. This is complemented by an impredicative part (due to the occurrence of parameter-free $\Pi^1_1$-comprehension), which we provide in Section~\ref{sect:impred-ord-ana}.

\section{Ordinal notations}\label{sect:ordinal-notations}

In this section, we define the order~$\psi(\Gamma_{\Omega+X})$ that was mentioned in the intro\-duction. The traditional approach would be to construct it via a single ordinal notation system (see Section~2 of~\cite{rathjen-model-bi} for a closely related case). We find it easier to separate the construction into two steps. First, we recall the notation system~$\Gamma_Y$ from~\cite{rathjen-atr}, which relativizes classical notations for the first fixed point~$\Gamma_0$ of the Veblen hierarchy. Secondly, we define $\psi(\Gamma_{\Omega+X})$ as the $1$-fixed point of the transformation~$Y\mapsto\Gamma_{Y+X}$ in the sense of~\cite{FR_Pi11-recursion}.

Before we define the ordinal notation system~$\Gamma_Y$, we give an informal motivation in set-theoretic terms. The so-called Veblen functions on the ordinal numbers are given by~$\varphi_0(\gamma)=\omega^\gamma$ and, when we have $\beta>0$, by
\begin{equation*}
\varphi_{\beta}(\gamma)=\text{``the $\gamma$-th joint fixed point of~$\varphi_\alpha$ for~$\alpha<\beta$"}.
\end{equation*}
To justify the definition, we point out that the~$\varphi_\alpha$ are normal functions, i.e., that they are strictly increasing and continuous at limits -- or equivalently, that they enumerate closed and unbounded (club) classes of ordinals. It is not hard to derive that $\varphi_\alpha$ has a club class of fixed points, and the intersection for~$\alpha<\beta$ is still club. We also define
\begin{equation*}
    \Gamma_\beta=\text{``the $\beta$-th fixed point of~$\alpha\mapsto\varphi_\alpha(0)$"},
\end{equation*}
where the indicated class of fixed points is club as a diagonal intersection. In order to give some intuition for the size of values, we note that $\varphi_1(0)=\varepsilon_0$ and $\Gamma_0$ are the proof-theoretic ordinals of Peano arithmetic and of~$\mathsf{ATR}_0$, respectively.

Values of the Veblen functions can be compared according to
\begin{equation*}
    \varphi_\alpha(\gamma)<\varphi_\beta(\delta)\quad\Leftrightarrow\quad\begin{cases}
        \gamma<\varphi_\beta(\delta)&\text{when }\alpha<\beta,\\
        \gamma<\delta&\text{when }\alpha=\beta,\\
        \varphi_\alpha(\gamma)<\delta & \text{when }\alpha>\beta.
    \end{cases}
\end{equation*}
For $\alpha<\beta$, the equivalence holds because $\varphi_\alpha$ is strictly increasing and because we have $\varphi_\beta(\delta)=\varphi_\alpha(\varphi_\beta(\delta))$, as $\varphi_\beta$ enumerates fixed points of~$\varphi_\alpha$. At least intuitively, the equivalence allows for recursive comparisons, as e.g.\ the expression $\gamma<\varphi_\beta(\delta)$ is shorter than $\varphi_\alpha(\gamma)<\varphi_\beta(\delta)$. Also recall that a non-zero ordinal has the form~$\omega^\gamma$ precisely if it is additively indecomposable (by convention, in Definition \ref{def:Gamma} we will use $H$ in reference to the German word ``Hauptzahlen"). This yields
\begin{equation*}
    \alpha+\beta<\varphi_\gamma(\delta)\quad\Leftrightarrow\quad\alpha,\beta<\varphi_\gamma(\delta).
\end{equation*}
The indicated properties of the Veblen functions motivate the following definition (see below for additional explanations). Our notation system coincides with the one from Definition~4.5 of~\cite{FR_Pi11-recursion} (see also the earlier construction in~\cite{rathjen-atr}). It is a straightforward extension of classical notation systems for~$\Gamma_0$.

\begin{definition}\label{def:Gamma}
    Given a linear order~$Y$, we use simultaneous recursion to define term systems~$H\subseteq\Gamma_Y$, a function~$v:\Gamma_Y\to\Gamma_Y$ and a binary relation~$\prec$ on $\Gamma_Y$ (where $\alpha\preceq\beta$ will mean that we have $\alpha\prec\beta$ or $\alpha=\beta$ as terms):
    \begin{enumerate}[label=(\roman*)]
    \item We have an element $0\in\Gamma_Y\backslash H$ with $v(0)=0$.
    \item For each~$x\in Y$, we have an element~$\Gamma_x\in H$ with $v(\Gamma_x)=\Gamma_x$.
    \item Consider terms $\alpha,\gamma\in\Gamma_Y$ with $v(\gamma)\preceq\alpha$, where we require~$\gamma\neq 0$ if~$\alpha$ is of the form~$\Gamma_x$. We then get another term~$\overline\varphi_\alpha(\gamma)\in H$ with~$v(\overline\varphi_\alpha(\gamma))=\alpha$.
    \item Given $n\geq 2$ terms~$\alpha_i\in H$ with $\alpha_1\succeq\ldots\succeq\alpha_n$, we get a term~$\alpha_1+\ldots+\alpha_n$ in $\Gamma_Y\backslash H$ with $v(\alpha_1+\ldots+\alpha_n)=0$.
    \end{enumerate}
    We write any $\alpha\in\Gamma_Y$ as $\alpha=\alpha_1+\ldots+\alpha_n$ with $\alpha_i\in H$, where $n=0$ and $n=1$ correspond to~$\alpha=0$ and~$\alpha\in H$, respectively. The relation~${\prec}\subseteq\Gamma_Y\times\Gamma_Y$ is defined by the following closure properties:
    \begin{enumerate}[label=(\roman*')]
    \item Given~$x<y$ in~$Y$, we get $\Gamma_x\prec\Gamma_y$.
    \item We obtain $\overline\varphi_\alpha(\gamma)\prec\overline\varphi_\beta(\delta)$ in each of the following cases:
    \begin{itemize}[label=--]
        \item we have $\alpha\prec\beta$ and $\gamma\prec\overline\varphi_\beta(\delta)$,
        \item we have $\alpha=\beta$ and $\gamma\prec\delta$,
        \item we have $\overline\varphi_\alpha(\gamma)\preceq\delta$.
    \end{itemize}
    \item We get $\overline\varphi_\alpha(\gamma)\prec\Gamma_x$ when we have $\alpha,\gamma\prec\Gamma_x$, and we get $\Gamma_x\prec\overline\varphi_\alpha(\gamma)$ when we have~$\Gamma_x\preceq\alpha$ or $\Gamma_x\preceq\gamma$.
    \item We obtain $\alpha_1+\ldots+\alpha_m\prec\beta_1+\ldots+\beta_n$ in each of the following cases (where any $m,n\geq 0$ are admitted):
    \begin{itemize}[label=--]
    \item we have $m<n$ and $\alpha_i=\beta_i$ for all~$i\leq m$,
    \item there is a $j\leq\min(m,n)$ with $\alpha_j\prec\beta_j$ and $\alpha_i=\beta_i$ for all~$i<j$.
    \end{itemize}
    \end{enumerate}
    Having completed the construction of the term sytem~$\Gamma_Y$, we now declare that the function $\varphi:\Gamma_Y\times\Gamma_Y\to H\subseteq\Gamma_Y$ is given by
    \begin{equation*}
        \varphi_\alpha(\gamma)=\begin{cases}
            \alpha & \text{if $\alpha$ is of the form~$\Gamma_x$ and $\gamma=0$},\\
            \gamma & \text{if $\alpha\prec v(\gamma)$},\\
            \overline\varphi_\alpha(\gamma) & \text{otherwise}.
        \end{cases}
    \end{equation*}
Note that the third clause applies precisely when~$\overline\varphi_\alpha(\gamma)$ is a term of~$\Gamma_Y$.
\end{definition}

If one wants to disentangle the simultaneous recursion in the previous definition, one can first generate a larger set of terms by omitting the conditions~$v(\gamma)\preceq\alpha$ and~$\alpha_1\succeq\ldots\succeq\alpha_n$ in~(iii) and~(iv). In this larger set, one then decides~$\alpha\prec\beta$ by recursion on the combined complexity of~$\alpha$ and~$\beta$. Once~$\prec$ is known, one re\-introduces the omitted conditions in~(iii) and~(iv) to determine the terms in~$\Gamma_Y$.

In clause~(ii') of Definition~\ref{def:Gamma}, the reader may have expected that the third condition demands not just $\overline\varphi_\alpha(\gamma)\preceq\delta$ but rather $\overline\varphi_\alpha(\gamma)\prec\delta$ and additionally~$\alpha\succ\beta$, as in Definition~2.5 of~\cite{rathjen-atr}. By Section~4 of~\cite{FR_Pi11-recursion}, the two conditions yield the same linear order, but our version facilitates the proof of the next lemma below. Crucially, Lemma~4.6(a)~of~\cite{FR_Pi11-recursion} shows that we have $\alpha,\gamma\prec\overline\varphi_\alpha(\gamma)$ when~$\overline\varphi_\alpha(\gamma)$ is a term of~$\Gamma_Y$. This relates to the requirement $v(\gamma)\preceq \alpha$ in condition~(iii) of Definition~\ref{def:Gamma}, which prevents redundant terms from entering the notation system. Specifically, a term $\overline \varphi_\alpha(\overline \varphi_\beta(\gamma))$ would be redundant for ${\alpha\prec\beta}$ due to $\varphi_\alpha\circ \varphi_\beta=\varphi_\beta$. One can view~$\overline\varphi$ as a partial function that is undefined where~$\varphi$ has fixed points. This has the effect that term notations are unique, which is required if identity of terms is to provide the right notion of equality. The following standard result is proved, e.g., as Lemma~4.6(b)~of~\cite{FR_Pi11-recursion}.

\begin{lemma}[$\mathsf{RCA}_0$]\label{lem:Gamma-lin}
If $Y$ is a linear order, so is $(\Gamma_Y,\prec)$.
\end{lemma}

The following proposition shows that our notation system reflects central properties of the set-theoretic construction above, which provides further justification for Definition~\ref{def:Gamma}. For a comprehensive list of other fundamental properties, we refer to Proposition~4.13 of~\cite{FR_Pi11-recursion}.

\begin{proposition}[$\mathsf{RCA}_0$]\label{prop:Veblen-order}
(a) We have $\alpha,\gamma\preceq\varphi_\alpha(\gamma)$.

(b) For $\gamma\prec\delta$ we have $\varphi_\alpha(\gamma)\prec\varphi_\alpha(\delta)$.

(c) Given $\alpha\prec\beta$, we get
\begin{equation*}
\varphi_\alpha(\gamma)\prec\varphi_\beta(\delta)\quad\Leftrightarrow\quad\gamma\prec\varphi_\beta(\delta).
\end{equation*}
\end{proposition}
\begin{proof}
(a) We distinguish cases according to the definition of~$\varphi$. If we have $\alpha=\Gamma_x$ and $\gamma=0$, we get $\varphi_\alpha(\gamma)=\alpha$ and the claim is immediate. Now assume that we have~$\alpha\prec v(\gamma)$ and hence~$\varphi_\alpha(\gamma)=\gamma$. The task is to show~$\alpha\preceq\gamma$. We have already mentioned that $\beta,\delta\prec\overline\varphi_\beta(\delta)$ holds whenever~$\overline\varphi_\beta(\delta)$ is a term of~$\Gamma_Y$, as shown in~\cite{FR_Pi11-recursion}. It is straightforward to see that this entails~$v(\gamma)\preceq\gamma$. So we indeed get~$\alpha\prec\gamma$. In the remaining case, we have $\alpha,\gamma\prec\overline\varphi_\alpha(\gamma)=\varphi_\alpha(\gamma)$.

(b) First assume that we have~$\varphi_\alpha(\gamma)=\alpha$. Due to~(a), we obtain~$\varphi_\alpha(\gamma)\preceq\varphi_\alpha(\delta)$. Aiming at a contradiction, we assume that the latter is an equality. In view of~$\delta\neq 0$ and $\overline\varphi_\alpha(\delta)\neq\alpha$, we must then have~$\varphi_\alpha(\delta)=\delta$ with $\alpha\prec v(\delta)$. We thus get $\delta=\alpha$ and hence~$\delta\prec v(\delta)$, which goes against an observation from the proof of~(a). If we have $\varphi_\alpha(\gamma)=\gamma$, we can use~(a) to derive $\varphi_\alpha(\gamma)\prec\delta\preceq\varphi_\alpha(\delta)$. Finally, assume that we have $\varphi_\alpha(\gamma)=\overline\varphi_\alpha(\gamma)$. If we also have $\varphi_\alpha(\delta)=\overline\varphi_\alpha(\delta)$, the claim holds by clause~(ii') of Definition~\ref{def:Gamma}. In view of~$\delta\neq 0$, it only remains to consider the case where we have~$\varphi_\alpha(\delta)=\delta$ with $\alpha\prec v(\delta)$. In view of $v(\delta)\neq 0$, we have two subcases. First assume that $\delta$ is of the form~$\Gamma_y$. Given~$\alpha,\gamma\prec\delta$, clause~(iii') yields the claim. Finally, if we have~$\delta=\overline\varphi_{v(\delta)}(\eta)$, we can again conclude by clause~(ii').

(c) It is straightforward to check $v(\varphi_\beta(\delta))\succeq\beta$ by case distinction on the definition of~$\varphi$. Given~$\alpha\prec\beta$, we thus get $\varphi_\alpha(\varphi_\beta(\delta))=\varphi_\beta(\delta)$. The desired equivalence follows by part~(b) and the fact that $\prec$ is linear.
\end{proof}

In addition to the function $\varphi$, we will need a few other operations on the term system $\Gamma_Y$.

\begin{definition}\label{def:Gamma-operations}
    To define an addition operation on~$\Gamma_Y$, we recall that elements of~$\Gamma_Y$ may be uniquely written as $\alpha_1+\ldots+\alpha_m$ with $\alpha_i\in H$ and $\alpha_1\succeq\ldots\succeq\alpha_m$, where $m=0,1$ is permitted. Using this representation, we set
\begin{equation*}
    (\alpha_1+\ldots+\alpha_m)+(\beta_1+\ldots+\beta_n)=\alpha_1+\ldots+\alpha_i+\beta_1+\ldots+\beta_n
\end{equation*}
    for the maximal~$i\leq m$ with $\alpha_i\succeq\beta_1$, where we put $i=0$ in case we have~$\alpha_1\prec\beta_1$ and $i=m$ if we have~$n=0$. We adopt the notation
    \begin{equation*}
    n:=\underbrace{\varphi_0(0) + \dots + \varphi_0(0)}_{n\, \text{times}}\quad\text{and}\quad \omega:=\varphi_0(1).
    \end{equation*}
    Multiplication by a natural number $n\prec \omega$ is defined by
        \begin{equation*}
        \alpha\cdot n:= \underbrace{\alpha + \dots + \alpha}_{n\, \text{times}}.
    \end{equation*}
    To define multiplication by $\omega$ from the left, we stipulate
    \begin{multline*}
        \quad\omega\cdot\alpha:=\varphi_0(1+\beta_1)+\ldots+\varphi_0(1+\beta_m)\\
        \text{for }\alpha=\alpha_1+\ldots+\alpha_m\text{ and }\beta_i=\begin{cases}
            \beta & \text{if }\alpha_i=\overline\varphi_0(\beta),\\
            \alpha_i & \text{otherwise}.
        \end{cases}
    \end{multline*}
\end{definition}

In the definition of multiplication, we always have $\alpha_i=\varphi_0(\beta_i)$. With the more familiar notation $\omega^\beta=\varphi_0(\beta)$, the definition thus amounts to
\begin{equation*}
    \omega\cdot\left(\omega^{\beta_1}+\ldots+\omega^{\beta_m}\right)=\omega^{1+\beta_1}+\ldots+\omega^{1+\beta_m}.
\end{equation*}
 It is straightforward to see that $\alpha,\beta\prec\gamma\in H$ implies~$\alpha+\beta\prec\gamma$. Other basic facts are provided by Lemma~4.15 of~\cite{FR_Pi11-recursion}.

In the rest of this section, we construct the order~$\psi(\Gamma_{\Omega+X})$ that was mentioned in the introduction. Here~$X$ is a fixed linear order. For linear orders~$X_0$ and $X_1$, we define~$X_0+X_1$ as the linear order with elements~$(i,x)$ for~$x\in X_i$, where the maps~$X_i\ni x\mapsto(i,x)$ are embeddings and we have $(0,x)<(1,x')$ for all~$x\in X_0$ and~$x'\in X_1$. We will write $x$ and $X_0+x'$ at the place of~$(0,x)$ and~$(1,x')$. Intuitively, our aim is to find an order~$\Omega$ with a function
\begin{equation*}
    \overline\psi:\Gamma_{\Omega+X}\to\Omega
\end{equation*}
that preserves the order `as much as possible'. Assuming that $X$ is a non-empty well-order, $\Omega$ would have to be ill-founded if~$\overline\psi$ was fully order preserving (since the order-type of~$\Gamma_{\Omega+X}$ exceeds the one of~$\Omega$). The idea is that a slightly weaker condition on~$\overline\psi$ allows~$\Omega$ to be well-founded but forces it to be very large. We will eventually (see Corollary~\ref{cor:psi-order-pres} below) arrive at a condition of the form
\begin{equation*}
    \alpha\prec\beta\text{ and }\alpha\in C(\beta)\quad\Rightarrow\quad\overline\psi(\alpha)<\overline\psi(\beta).
\end{equation*}
Here $C(\beta)$ has good closure properties (see Lemma~\ref{lem:C-sets-closure} and Proposition~\ref{prop:C-sets-psi}(a)), which ensure that it contains many relevant~$\alpha$.

In our construction, $\Omega$ will be a certain countable order. There is a related set-theoretic construction that interprets~$\Omega$ as the least uncountable cardinal (see~\cite{buchholz-new-system} or Section~5.3.1 of~\cite{rathjen-sieg-stanford}). This construction defines~$C(\beta)$ and~$\overline\psi(\beta)$ by a simultaneous recursion on~$\beta$. In the recursion step, one first defines~$C(\beta)$ as the smallest set with the aforementioned closure properties (where closure as in Proposition~\ref{prop:C-sets-psi}(a) refers to the recursively constructed values~$\overline\psi(\alpha)$ with $\alpha\prec\beta$). The resulting set~$C(\beta)$ is easily seen to be countable. One then defines~$\overline\psi(\beta)$ as the (countable) supremum of the ordinals~$\overline\psi(\alpha)$ with $\alpha\prec\beta$ and $\alpha\in C(\beta)$. 

In order to carry out the indicated construction, we need to know that each element of~$\Gamma_{Y+X}$ depends on just finitely many elements of~$Y$, as expressed by the following definition and lemma.  Let us agree that $[Y]^{<\omega}$ denotes the set of all finite subsets of a set or order~$Y$.

\begin{definition}\label{def:Gamma-supp}
For a linear order~$Y$, define~$\supp_Y:\Gamma_{Y+X}\to[Y]^{<\omega}$ recursively~by
\begin{gather*}
    \supp_Y(\Gamma_y)=\{y\},\quad\supp_Y(\Gamma_{Y+x})=\emptyset,\\
    \supp_Y(\overline\varphi_\alpha(\beta))=\supp_Y(\alpha)\cup\supp_Y(\beta),\\
    \supp_Y(\alpha_1+\ldots+\alpha_n)=\textstyle\bigcup_{i\leq n}\supp_Y(\alpha_i)\text{ for }\alpha_1\succeq\ldots\succeq\alpha_n\text{ in H}.
\end{gather*}
Note that the last clause for~$n=0$ yields~$\supp_Y(0)=\emptyset$.
\end{definition}

For a more general version and a proof of the following result, we refer the reader to Proposition~4.8 of~\cite{FR_Pi11-recursion}.

\begin{lemma}[$\mathsf{RCA}_0$]\label{lem:Gamma-funct}
    If $Y$ is a suborder of~$Z$, we have the following:

    (a) $\Gamma_{Y+X}$ is a suborder of~$\Gamma_{Z+X}$,

    (b) for $\alpha\in\Gamma_{Y+X}$ we have $\supp_Y(\alpha)=\supp_Z(\alpha)$,

    (c) any $\alpha\in\Gamma_{Z+X}$ with $\supp_Z(\alpha)\subseteq Y$ lies in~$\Gamma_{Y+X}$.
\end{lemma}

The previous result effectively shows that $\Gamma_{(\cdot)+X}$ is a dilator. This terminology will not be used in the present paper. We only need one case of a general construction on dilators, which is given by the following definition. Existence and uniqueness (up to isomorphism) of the objects in the definition are established in Section~2 of~\cite{FR_Pi11-recursion} (specialized to~$\nu=1$ and~$D(Y)=\Gamma_{Y+X}$), where $\Omega$ and $\pi$ are constructed via a computable system of terms (so that they are available in~$\mathsf{RCA}_0$). Note that we write $\rng(g)=\{g(x):x\in X\}$ for the range of~$g:X\to Y$.

\begin{definition}\label{def:Gamma-Omega}
Consider a linear order~$X$ (as fixed throughout this section). We characterize~$\Omega=\Omega(X)$ and~$\pi:\Omega\to\Gamma_{\Omega+X}$ as the (up to isomorphism unique) linear order and embedding that validate
    \begin{equation*}
        \rng(\pi)=\{\alpha\in\Gamma_{\Omega+X}:\pi(s)\prec\alpha\text{ for all }s\in\supp_\Omega(\alpha)\}
    \end{equation*}
and admit an $h:\Omega\to\mathbb N$ with $h(s)<h(t)$ for all $s\in\supp_\Omega(\pi(t))$.
\end{definition}

The definition is gleaned from traditional notation systems for the Bachmann-Howard ordinal, which involve a partial collapsing function that is the inverse of~$\pi$ (see~\cite{buchholz-bachmann-howard,rathjen-sieg-stanford} for context). It may first seem ad hoc, but we now show how to extract the collapsing structure that was motivated above. The function~$h$ gives us a substitute for the height of terms. When $\Gamma_{\Omega+X}$ is well-founded (as asserted by our well-ordering principle), the existence of~$h$ is automatic (over a suitable base theory). Indeed, we can then define $h(t)=\sup\{h(s)+1:s\in\supp_\Omega(\pi(t))\}$ by recursion on the value~$\pi(t)\in\Gamma_{\Omega+X}$, as the condition on~$\rng(\pi)$ yields~$\pi(s)\prec\pi(t)$.

\begin{definition}\label{def:psi}
    Let $\overline\psi:\Gamma_{\Omega+X}\to\Omega$ be given by
    \begin{equation*}
        \overline\psi(\alpha)=\begin{cases}
            s & \text{if }\alpha=\pi(s),\\
            \min_\Omega\{s\in\supp_\Omega(\alpha):\pi(s)\succeq\alpha\} & \text{if }\alpha\notin\rng(\pi).
        \end{cases}
    \end{equation*}
    Note that the minimum is taken over a non-empty finite set for $\alpha\notin\rng(\pi)$, by the condition from Definition~\ref{def:Gamma-Omega}. We write
    \begin{equation*}
      \psi(\Gamma_{\Omega+X})=\{\beta\in\Gamma_{\Omega+X}:\beta\prec\Gamma_{\Omega+x}\text{ for all }x\in X\}
    \end{equation*}
    and define~$\psi:\Gamma_{\Omega+X}\to\psi(\Gamma_{\Omega+X})$ by $\psi(\alpha)=\Gamma_{\overline\psi(\alpha)}$.
\end{definition}

The advantage of~$\psi$ over~$\overline\psi$ is that it maps~$\Gamma_{\Omega+X}$ into itself and that its values are terms with strong closure properties (in the order given by Definition~\ref{def:Gamma}). We now define the sets that were mentioned in the informal discussion before Definition~\ref{def:Gamma-supp}.

\begin{definition}\label{def:C-psi}
 For~$\alpha\in\Gamma_{\Omega+X}$, we set
    \begin{equation*}
    C(\alpha)=\{\gamma\in\Gamma_{\Omega+X}:\pi(s)\prec\alpha\text{ for all }s\in\supp_\Omega(\gamma)\}.
    \end{equation*}
\end{definition}

Let us note that~$\alpha\in\rng(\pi)$ is equivalent to~$\alpha\in C(\alpha)$. As promised, we now show that~$C(\alpha)$ has good closure properties. The following lemma is immediate in view of Definition~\ref{def:Gamma-supp}. Part~(c) refers to addition in the sense of Definition~\ref{def:Gamma-operations}. We note that part~(a) does, in particular, yield~$0\in C(\gamma)$.

\begin{lemma}[$\mathsf{RCA}_0$]\label{lem:C-sets-closure}
    (a) For $\alpha_1\succeq\ldots\succeq\alpha_n$ in~$H$, we have
    \begin{equation*}
        \alpha_1+\ldots+\alpha_n\in C(\gamma)\quad\Leftrightarrow\quad\{\alpha_1,\ldots,\alpha_n\}\subseteq C(\gamma).
    \end{equation*}

    (b) A term~$\overline\varphi_\alpha(\beta)$ lies in~$C(\gamma)$ precisely when we have~$\alpha,\beta\in C(\gamma)$.
    
    (c) Given $\alpha,\beta\in C(\gamma)$, we get $\alpha+\beta\in C(\gamma)$ and $\varphi_\alpha(\beta)\in C(\gamma)$.

    (d) We have $\Gamma_{\Omega+x}\in C(\gamma)$ for any~$x\in X$ and any~$\gamma\in\Gamma_{\Omega+X}$.    
\end{lemma}

Further closure properties can be derived. We work out the following case. The reader may verify other basic closure properties as they occur.

\begin{remark}\label{rmk:omega*alpha-closure}
    For given $\alpha\in C(\gamma)$, we obtain $\omega\cdot\alpha\in C(\gamma)$. To see this, let us write $\alpha=\varphi_0(\beta_1)+\ldots+\varphi_0(\beta_n)$ as in Definition~\ref{def:Gamma-operations}. By the lemma above, we learn that $\varphi_0(\beta_i)\in C(\gamma)$ holds for each~$i$. Now $\varphi_0(\beta_i)$ is either $\beta_i$ or $\overline\varphi_0(\beta_i)$. So we can conclude $\beta_i\in C(\gamma)$ by the lemma. The latter then yields
    \begin{equation*}
        \omega\cdot\alpha=\varphi_0(1+\beta_1)+\ldots+\varphi_0(1+\beta_n)\in C(\gamma),
    \end{equation*}
    as desired.
\end{remark}

The following is a special case of Proposition~7.5 from~\cite{FR_Pi11-recursion}.

\begin{proposition}[$\mathsf{RCA}_0$]\label{prop:C-sets-psi}
    (a) Given $\alpha\in C(\beta)$ and $\alpha\prec\beta$, we get $\psi(\alpha)\in C(\beta)$.

    (b) From $\alpha\in C(\beta)\cap\psi(\Gamma_{\Omega+X})$, we obtain~$\alpha\prec\psi(\beta)$.

    (c) If we have $\beta\in C(\beta)$, then $\alpha\prec\psi(\beta)$ implies~$\alpha\in C(\beta)$.
\end{proposition}

Clearly, the codomain $\psi(\Gamma_{\Omega+X})$ of~$\psi$ is an initial segment of~$\Gamma_{\Omega+X}$. By combining parts~(b) and~(c), we thus obtain
\begin{equation*}
C(\beta)\cap\psi(\Gamma_{\Omega+X})=\{\alpha\in\Gamma_{\Omega+X}:\alpha\prec\psi(\beta)\}\quad\text{when}\quad\beta\in C(\beta).
\end{equation*}

\begin{proof}
    (a) Let us put
    \begin{equation*}
        \alpha^+=\begin{cases}
            \alpha & \text{if }\alpha\in C(\alpha),\\
            \pi(s) & \text{for the $<_{\Omega}$-minimal $s\in\supp_\Omega(\alpha)$ with $\pi(s)\succeq\alpha$ if~$\alpha\notin C(\alpha)$}.
        \end{cases}
    \end{equation*}
    One can verify that for any $\alpha\in\Gamma_{\Omega+X}$ and $s\in\Omega$ we have
    \begin{equation*}
        \psi(\alpha)=\Gamma_s\quad\Leftrightarrow\quad\pi(s)=\alpha^+.
    \end{equation*}
    The left side yields~$\supp_\Omega(\psi(\alpha))=\{s\}$, so that the desired conclusion $\psi(\alpha)\in C(\beta)$ is equivalent to~$\alpha^+\prec\beta$. We have
    \begin{equation*}
        \alpha^+\in\{\alpha\}\cup\{\pi(r):r\in\supp_\Omega(\alpha)\}.
    \end{equation*}
    By the assumptions of~(a), all elements of the set on the right are bounded by~$\beta$.

    (b) Let us first assume that $\alpha$ is of the form~$\Gamma_s$. Here we have $\supp_\Omega(\alpha)=s$, so that the assumption~$\alpha\in C(\beta)$ amounts to~$\pi(s)\prec\beta$. We now write $\psi(\beta)=\Gamma_t$ with $\pi(t)=\beta^+\succeq\beta$. Given that~$\pi$ is an order embedding, we get $s<t$ and thus indeed~$\alpha\prec\psi(\beta)$. To obtain the general case, we argue by induction on the build-up of the term~$\alpha$. First assume that the latter has the form~$\overline\varphi_\gamma(\delta)$. Using the previous lemma, we get $\gamma,\delta\in C(\beta)\cap\psi(\Gamma_{\Omega+X})$, so that the induction hypothesis yields~$\gamma,\delta\prec\psi(\beta)$. Given that $\psi(\beta)$ has the form~$\Gamma_t$, this entails~$\alpha\prec\psi(\beta)$. The argument for~$\alpha$ of the form~$\alpha_1+\ldots+\alpha_n$ is similar.

    (c) Again, we first assume $\alpha=\Gamma_s$ and write $\beta=\Gamma_t$ with $\pi(t)=\beta^+$. Due to the assumption~$\beta\in C(\beta)$, we have $\beta^+=\beta$ (cf.~the comment after Definition~\ref{def:C-psi}). Given $\alpha\prec\psi(\beta)$, we thus get~$s<t$ and hence~$\pi(s)\prec\beta$. In view of~$\supp_\Omega(\alpha)=\{s\}$, this shows~$\alpha\in C(\beta)$. For the general case, one again argues by induction on the build-up of~$\alpha$, using the previous lemma.
\end{proof}

The following corollary confirms that we get an `almost' order preserving collapsing function, as needed for the impredicative ordinal analysis that will be carried out in Section~\ref{sect:impred-ord-ana}. Note that $\alpha\prec\beta$ entails~$C(\alpha)\subseteq C(\beta)$. Thus the assumption of the corollary is satisfied for~$\alpha\in C(\alpha)$, which is the condition in Lemma~1.3 of~\cite{buchholz-new-system}.

\begin{corollary}[$\mathsf{RCA}_0$]\label{cor:psi-order-pres}
    If we have $\alpha\in C(\beta)$, then $\alpha\prec\beta$ implies~$\psi(\alpha)\prec\psi(\beta)$.
\end{corollary}

For the ordinal analysis in later sections, it is fruitful to recast the previous considerations in terms of so-called operators (as introduced by Buchholz~\cite{buchholz-local-predicativity}).

\begin{definition}\label{def:H-operator}
    For each $\gamma\in\Gamma_{\Omega+X}$ and any finite~$a,b\subseteq\Gamma_{\Omega+X}$, we put
    \begin{equation*}
        \mathcal H_\gamma[a](b)=\bigcap\{C(\delta):\gamma\prec\delta\text{ and }a\cup b\subseteq C(\delta)\}\subseteq\Gamma_{\Omega+X}.
    \end{equation*}
    By an operator, we mean a function of the form~$b\mapsto\mathcal H_\gamma[a](b)$. When we have $a=\emptyset$, we also write $\mathcal H_\gamma$ at the place of~$\mathcal H_\gamma[a]$.
\end{definition}

In the literature, one typically allows~$a$ and~$b$ to be infinite. This has the (purely) aesthetic effect that operators have the same domain and codomain (the power set of~$\Gamma_{\Omega+X}$). The restriction to finite sets allows us to construct operators over~$\mathsf{RCA}_0$, as Definition~\ref{def:C-psi} yields $\mathcal H_\gamma[a](b)=C(\delta)$ for
\begin{equation*}
    \delta=\max\left(\{\gamma+1\}\cup\{\pi(s)+1:s\in\supp_\Omega(\eta)\text{ for some }\eta\in a\cup b\}\right).
\end{equation*}
That $\mathcal H_\gamma[a](b)$ is equal to a single set $C(\delta)$ with $\delta\succ\gamma$ (rather than an intersection of such sets) will also simplify some proofs (though only slightly). In particular, we get the following.

\begin{remark}\label{rmk:H-closure}
    All values $\mathcal H(b)$ of our operators enjoy the closure properties from Lemma~\ref{lem:C-sets-closure} (i.e., the latter remains valid when $C(\gamma)$ is replaced by~$\mathcal H(b)$). This means that our operators are nice in the sense of~\cite{buchholz-local-predicativity}.
\end{remark}

The following lemma shows that we are indeed concerned with closure operators. By combining (a) and~(b), we see that $b\subseteq c$ implies~$\mathcal H(b)\subseteq\mathcal H(c)$.

\begin{lemma}[$\mathsf{RCA}_0$]\label{lem:closure-op}
Consider any operator $\mathcal H$ and finite $b,c\subseteq\Gamma_{\Omega+X}$.

    (a) We have $b\subseteq\mathcal H(b)$.
    
    (b) From $b\subseteq\mathcal H(c)$ we get $\mathcal H(b)\subseteq\mathcal H(c)$.
\end{lemma}
\begin{proof}
    (a) Writing $\mathcal H=\mathcal H_\gamma[a]$, Definition~\ref{def:H-operator} gives $b\subseteq a\cup b\subseteq\mathcal H_\gamma[a](b)=\mathcal H(b)$.

    (b) With $\mathcal H=\mathcal H_\gamma[a]$, we get $a\cup b\subseteq\mathcal H(c)=\mathcal H_\gamma[a](c)=C(\delta)$ for some~$\delta\succ\gamma$. The intersection from Definition~\ref{def:H-operator} yields $\mathcal H(b)=\mathcal H_\gamma[a](b)\subseteq C(\delta)=\mathcal H(c)$.
\end{proof}

We will employ the following construction on operators.

\begin{definition}\label{def:operator-variant}
    For any operator~$\mathcal H$, we define $\mathcal H[b]$ by $\mathcal H[b](c)=\mathcal H(b\cup c)$. When we have $b=\{\beta\}$, we also write $\mathcal H[\beta]$ at the place of~$\mathcal H[b]$.
\end{definition}

Note that $\mathcal H=\mathcal H_\gamma[a]$ entails $\mathcal H[b](c)=\mathcal H_\gamma[a](b\cup c)=\mathcal H_\gamma[a\cup b](c)$. It follows that $\mathcal H[b]=\mathcal H_\gamma[a\cup b]$ is again an operator in the sense of Definition~\ref{def:H-operator}. In particular, we get $\mathcal H_\gamma[\emptyset][a]=\mathcal H_\gamma[a]$, which defuses the slight ambiguity that arises from the abbreviation $\mathcal H_\gamma=\mathcal H_\gamma[\emptyset]$.

\begin{corollary}[$\mathsf{RCA}_0$]\label{cor:H[b]=H}
    For $b\subseteq\mathcal H(\emptyset)$ we have $\mathcal H[b]=\mathcal H$.
\end{corollary}
\begin{proof}
    We get $b\cup c\subseteq\mathcal H(c)$ and hence $\mathcal H[b](c)=\mathcal H(b\cup c)\subseteq\mathcal H(c)$.
\end{proof}

Due to $\mathcal H_\gamma[a](b)=\mathcal H_\gamma[\emptyset](a\cup b)$, it is often enough to consider the operators $\mathcal H_\gamma$. The following is a direct consequence of Definition~\ref{def:H-operator}.

\begin{lemma}[$\mathsf{RCA}_0$]\label{lem:operators-monotone}
Given $\gamma\preceq\delta$ and $a\subseteq b$, we get $\mathcal H_\gamma(a)\subseteq\mathcal H_\delta(b)$.
\end{lemma}

To conclude this section, we record the important connection between operators and collapsing functions.

\begin{proposition}[$\mathsf{RCA}_0$]\label{prop:H-closure-psi}
(a) We have
\begin{align*}
    \gamma\preceq\delta\text{ and }\gamma\in\mathcal H_\delta(a)\quad&\Rightarrow\quad\psi(\gamma)\in\mathcal H_\delta(a),\\
    \gamma,\delta\prec\eta\text{ and }\delta\in\mathcal H_\gamma(a)\text{ for some }a\subseteq C(\eta)\quad&\Rightarrow\quad\psi(\delta)\prec\psi(\eta).
\end{align*}
(b) If we have~$\delta\in\mathcal H_\delta(\emptyset)$, then
\begin{equation*}
    \mathcal H_\delta(\emptyset)\cap\psi(\Gamma_{\Omega+X})=\{\gamma\in\Gamma_{\Omega+X}:\gamma\prec\psi(\delta+1)\}
\end{equation*}
is an initial segment of~$\Gamma_{\Omega+X}$.
\end{proposition}
\begin{proof}
    (a) For the first implication, recall $\mathcal H_\delta(a)=C(\xi)$ with $\xi\succ\delta$, which reduces the claim to Proposition~\ref{prop:C-sets-psi}(a). In the second implication, the premises $\gamma\prec\eta$ and $a\subseteq C(\eta)$ ensure $\mathcal H_\gamma(a)\subseteq C(\eta)$. The claim follows by Corollary~\ref{cor:psi-order-pres}.

    (b) By the paragraph after Definition~\ref{def:H-operator}, we have $\mathcal H_\delta(\emptyset)=C(\delta+1)$. Since we get $\delta+1\in C(\delta+1)$, the claim follows by Proposition~\ref{prop:C-sets-psi}.
\end{proof}

\section{A well-ordering proof}\label{sect:well-ordering-proof}

In this section, we prove that the order~$\psi(\Gamma_{\Omega+X})$ from the previous section is well-founded for any well-order~$X$. Our proof will be based on the principle~$\Pi^1_1\text{-}\mathsf{CA}^\Gamma_0$, which asserts that any set~$Z\subseteq\mathbb N$ is contained in an $\omega$-model that validates both $\Pi^1_1$-comprehension with no set parameters except for~$Z$ and arithmetical transfinite recursion with arbitrary parameters. We thus establish the direction from~(i) to~(ii) in our main Theorem~\ref{thm:main}.

Throughout the previous section, we have fixed a well-order~$X$. When we want to make the dependency more explicit, we write $\Omega(X)$ at the place of~$\Omega$. The following result shows that our well-ordering principle is robust.

\begin{proposition}[$\mathsf{RCA}_0$]\label{prop:wo-princ-robust}
If one of the transformations
\begin{equation*}
    X\mapsto\Gamma_{\Omega+X}\quad\text{and}\quad X\mapsto\psi(\Gamma_{\Omega+X})\quad\text{and}\quad X\mapsto\Omega(X)
\end{equation*}
preserves well-orders, then so do the other two.
\end{proposition}
\begin{proof}
    In view of the embedding and inclusion
    \begin{equation*}
        \Omega\ni s\mapsto\Gamma_s\in \psi(\Gamma_{\Omega+X})\subseteq\Gamma_{\Omega+X},
    \end{equation*}
    it suffices to show that $X\mapsto\Gamma_{\Omega+X}$ preserves well-orders if this holds for~$X\mapsto\Omega(X)$. Assuming $X$ is a well-order, we first infer that the same holds for $Z=\Omega(X)+X$ and then for $\Omega(Z)$. To conclude, we construct an embedding $e:\Gamma_Z\to\Omega(Z)$. As in Lemma~\ref{lem:Gamma-funct}(a), we get an inclusion~$\iota:\Gamma_Z\to\Gamma_{\Omega(Z)+Z}$ with $\iota(\Gamma_x)=\Gamma_{\Omega(Z)+x}$.
    In view of Definition~\ref{def:Gamma-supp}, we obtain $\supp_{\Omega(Z)}(\iota(\alpha))=\emptyset$ for any~$\alpha\in\Gamma_Z$, which entails that we have $\iota(\alpha)\in C(\gamma)$ for all~$\gamma\in\Gamma_{\Omega(Z)+Z}$ and in particular for $\gamma=\iota(\beta)$. Let us put $e(\alpha):=\overline\psi(\iota(\alpha))$. Due to Corollary~\ref{cor:psi-order-pres} and Definition~\ref{def:psi}, we obtain
    \begin{equation*}
        \alpha\prec\beta\quad\Rightarrow\quad\Gamma_{\overline\psi(\iota(\alpha))}=\psi(\iota(\alpha))\prec\psi(\iota(\beta))=\Gamma_{\overline\psi(\iota(\beta))}\quad\Rightarrow\quad e(\alpha)<e(\beta),
    \end{equation*}
    as desired.
\end{proof}

For a linear order~$Y$ and any $x\in Y$, we write
\begin{equation*}
    Y[x]=\{z\in Y:z<_Y x\}.
\end{equation*}
Let us also agree on the abbreviation 
\begin{equation*}
    \mathsf{WO}(Y)\quad \iff \quad\text{``$Y$ is a well-order''}.
\end{equation*}
The next result holds by Section~4.2 of~\cite{rathjen-martinlof1}, since induction for any second-order formula is available in $\omega$-models . We note that $\Gamma_Y$ is computable relative to~$Y$ and hence absolute between $\mathcal M$ and the ambient world.

\begin{lemma}[$\mathsf{ACA}_0$]\label{lem:Gamma-prog}
    Consider an $\omega$-model $\mathcal M\vDash\mathsf{ATR}_0$ and a linear order $Y\in\mathcal M$. We have
    \begin{equation*}
        \mathcal M\vDash\forall x\in Y\big(\forall z<_Y\!x\,\,\mathsf{WO}(\Gamma_{Y[z]})\,\rightarrow\,\mathsf{WO}(\Gamma_{Y[x]})\big).
    \end{equation*}
\end{lemma}

In the next result, it is crucial that~$X$ is a well-order in the ambient world and hence also but not only in~$\mathcal M$ (in contrast to the order~$W$ in the proof of Theorem~\ref{thm:WO_principle}).

\begin{corollary}[$\mathsf{ACA}_0$]
    For a well-order~$X$ and any $\omega$-model~$\mathcal M\vDash\mathsf{ATR}_0$ such that we have $X\in\mathcal M$, we get
    \begin{equation*}
        \mathcal M\vDash\forall Y\big(\mathsf{WO}(\Gamma_Y)\to\mathsf{WO}(\Gamma_{Y+X})\big).
    \end{equation*}
\end{corollary}
\begin{proof}
 Write $Z$ for the extension of~$X$ by a new largest element~$\top$. Due to the previous lemma, we can use arithmetical induction on~$z\in Z$ to establish
 \begin{equation*}
     \mathcal M\vDash\mathsf{WO}(\Gamma_{Y+Z[z]}).
 \end{equation*}
 The claim follows in view of $X=Z[\top]$.
\end{proof}

Finally, we can derive our well-ordering principle.

\begin{theorem}[$\Pi^1_1\textsf{-CA}^\Gamma_0$]\label{thm:WO_principle}
    If $X$ is a well-order, then so is $\psi(\Gamma_{\Omega+X})$.
\end{theorem}
\begin{proof}
    Towards a contradiction, we assume that there is a strictly decreasing sequence $f:\mathbb N\to\Gamma_{\Omega+X}$. By $\Pi^1_1\textsf{-CA}^\Gamma_0$, we find an $\omega$-model $\mathcal M\vDash\mathsf{ATR}_0$ that validates $\Pi^1_1$-comprehension with parameters~$X,f\in\mathcal M$ (which can be coded into a single parameter). As noted before Definition~\ref{def:Gamma-Omega}, the construction of~$\Omega$ from~$X$ is computable and hence absolute. We will establish $\mathcal M\vDash\text{``$\Gamma_{\Omega+X}$ is well-founded''}$. In view of $f\in\mathcal M$, this yields the desired contradiction.

    By $\Pi^1_1$-comprehension with parameter~$X$, we form the set
    \begin{equation*}
        W=\{s\in\Omega:\mathcal M\vDash\text{``\,$\Gamma_{\Omega[s]}$ is a well-order"}\}\in\mathcal M.
    \end{equation*}
    Due to the previous lemma and corollary, we get $\mathcal M\vDash\text{``\,$\Gamma_{W+X}$ is a well-order"}$. To complete the proof, it suffices to establish $W=\Omega$.
    
    Working in~$\mathcal M$, we use arithmetical induction over~$\tau\in\Gamma_{W+X}\subseteq\Gamma_{\Omega+X}$ to show that any $t\in\Omega$ with $\pi(t)=\tau$ satisfies~$t\in W$. Again by Lemma~\ref{lem:Gamma-prog}, it suffices to prove that $s\in W$ holds for all~$s<t$. We argue by side induction over~$h(s)\in\mathbb N$. For any~$r\in\supp_\Omega(\pi(s))$, Definition~\ref{def:Gamma-Omega} yields both $h(r)<h(s)$ and $\pi(r)\prec\pi(s)$, so that we get $r<s<t$ and then inductively~$r\in W$. Now Lemma~\ref{lem:Gamma-funct}(c) entails that we have~$\pi(s)\in\Gamma_{W+X}$. In view of $\pi(s)\prec\tau$, we can conclude $s\in W$ by the main induction hypothesis.

    Finally, we use induction over~$h(t)\in\mathbb N$ to show that $t\in W$ holds for all~$t\in\Omega$. As before, we inductively get $\supp_\Omega(\pi(t))\subseteq W$ and then $\pi(t)\in\Gamma_{W+X}$. We can infer $t\in W$ by the previous paragraph.
\end{proof}

In the previous proof, only the applications of Lemma~\ref{lem:Gamma-prog} and its corollary rely on the fact that we have $\mathcal M\vDash\mathsf{ATR}_0$. It is the latter that distinguishes $\Pi^1_1\text{-}\mathsf{CA}^\Gamma_0$ from weaker theories of partial impredicativity.

\section{Inductive definitions}\label{sect:ind-def}

We will use ordinal analysis to derive $\Pi^1_1\textsf{-CA}^\Gamma_0$ from our well-ordering principle. It is known that theories of inductive definitions can be used for a perspicuous ordinal analysis of both $\mathsf{ATR}_0$ and $\Pi^1_1$-comprehension~\cite{avigad-atr,bfps-inductive}. In this section, we adapt results from the literature to reformulate $\Pi^1_1\textsf{-CA}^\Gamma_0$ in terms of inductive definitions.

From this point on, we assume that all formulas are in negation normal form, i.e., that they are built from literals (negated and unnegated prime formulas) via the connectives~$\land,\lor$ and the quantifiers~$\forall,\exists$. Negation and implication are considered as defined operations that apply de Morgan's laws and eliminate double negations. For example, the expression $\neg\forall x(Qx\to\neg Qy)$ is not itself a formula but refers to the formula~$\exists x(Qx\land Qy)$.

We first consider a reformulation of $\Pi^1_1\textsf{-CA}^\Gamma_0$ within the language of second-order arithmetic. A formula~$\varphi$ (in negation normal form) is called $X$-positive for a set variable~$X$ if all free occurrences of~$X$ in~$\varphi$ are in subformulas~$t\in X$ rather than~$t\notin X$. When a formula~$\varphi(x,X)$ is $X$-positive, the implication
\begin{equation*}
X\subseteq Y\to\{x\in\mathbb N:\varphi(x,X)\}\subseteq\{x\in\mathbb N:\varphi(x,Y)\}
\end{equation*}
is logically valid. So positive formulas define monotone maps on the power set~$\mathcal P(\mathbb N)$, which are also called operators (but should not be confused with the operators of Definition~\ref{def:H-operator}). It is well-known that such operators have least fixed points. We consider an axiom schema~$\mathsf{FP}$ that asserts the existence of fixed points but does no require them to be least. This schema consists of the universal closures of all formulas
\begin{equation*}\tag{$\mathsf{FP}$}
\exists X\forall x(x\in X\leftrightarrow\varphi(x,X))
\end{equation*}
such that $\varphi$ is arithmetical and $X$-positive. Let us note that $\mathsf{FP}$ includes the principle of arithmetical comprehension (as the case where $\varphi$ does not contain~$X$ and is thus trivially $X$-positive). The following was shown by Avigad.

\begin{theorem}[$\mathsf{ACA}_0$; \cite{avigad-atr}]\label{thm:ATR-FP}
Arithmetical transfinite recursion is equivalent to~$\mathsf{FP}$.
\end{theorem}

Below, we discuss a reformulation of $\Pi^1_1\textsf{-CA}^\Gamma_0$ in a first-order language for inductive definitions. Whenever $\mathcal L$ is a first-order language, we write $\mathcal L^X$ for the extension by a fresh unary predicate symbol~$X$. We say that an $\mathcal L^X$-formula is positive when its negation normal form has no subformulas~$\neg Xt$.

Let $\mathcal L_0$ be the usual language of first-order arithmetic extended by a unary predicate symbol~$Q$ (which will allow us to relativize to a given set). Recursively, we define $\mathcal L_{n+1}$ as the extension of~$\mathcal L_n$ by a unary predicate symbol~$I_\varphi^n$ for each positive $\mathcal L_n^X$-formula~$\varphi$ with precisely one free variable. Finally, we set $\mathcal L_\omega=\bigcup_{n\in\mathbb N}\mathcal L_n$.

The condition that $\varphi$ has a unique free variable~$x$ is notationally convenient. For a formula~$\psi$, it allows us to define $\varphi[t,\lambda y\,\psi]$ as the formula that results from~$\varphi[x/t]$ when each occurrence of~$Xs$ is replaced by the substitution instance $\psi[y/s]$ (where variables are renamed to avoid clashes). When $\psi$ is of the form~$Py$ for a relation symbol~$P$, we write $\varphi[t,P]$ at the place of~$\varphi[t,\lambda y\,\psi]$. At the same time, the restriction to a single free variable is no real limitation: We will focus on $\omega$-models, where parameters can be denoted by numerals. In more general settings, one can simulate parameters via an encoding of tuples.

We define $\mathsf{ID}^\Gamma$ as the $\mathcal L_\omega$-theory that extends Peano arithmetic (with equality and induction axioms for all of~$\mathcal L_\omega$) by the following principles. First, each of our predicate symbols~$I_\varphi^n$ gives rise to a fixed point axiom~($\mathsf F$), which is given by
\begin{equation*}\tag{\textsf F}
\begin{array}{ll}
\forall x\big(\varphi[x,I_\varphi^n]\rightarrow I_\varphi^n x\big) &  \quad\text{when }n=0,\\[1.5ex]
\forall x\big(\varphi[x,I_\varphi^n]\leftrightarrow I_\varphi^n x\big) & \quad\text{when }n>0.
\end{array}
\end{equation*}
Secondly, we add a least fixed point axiom
\begin{equation*}\tag{\textsf L}
\forall x\big(\varphi[x,\lambda x\,\psi]\to\psi\big)\to\forall x\big(I_\varphi^0x\to\psi\big)
\end{equation*}
for each predicate symbol~$I_\varphi^0$ (note the restriction to $n=0$) and each $\mathcal L_\omega$-formula~$\psi$. To connect with the literature, we point out that $\mathsf{ID}^\Gamma$ is a version of $\widehat{\mathsf{ID}}_{<\omega}$ (see~\cite{avigad-atr,feferman-IDhat}) that is built on top of~$\mathsf{ID}_1$ (see~\cite{bfps-inductive,howard-ID1}). The case distinction in~($\mathsf F$) is explained by the following standard result.

\begin{lemma}\label{lem:F-equiv}
    For each predicate symbol~$I^0_\varphi$, we have
    \begin{equation*}
        \mathsf{ID}^\Gamma\vdash\forall x\big(I_\varphi^0 x\rightarrow\varphi[x,I_\varphi^0]\big).
    \end{equation*}
\end{lemma}
\begin{proof}
    Define $\psi$ as $\varphi[x,I^0_\varphi]$ in order to write ($\mathsf F$) as $\forall x(\psi\to I^0_\varphi x)$. By induction on the positive formula~$\varphi$, we obtain $\varphi[x,\lambda x\,\psi]\to\varphi[x,I^0_\varphi]$, which is $\varphi[x,\lambda x\,\psi]\to\psi$. We can conclude by~($\mathsf L$).
\end{proof}

Concerning the following result, we recall that $\omega$-models for the language of second-order arithmetic can be identified with subsets of~$\mathcal P(\mathbb N)$. By an $\omega$-model of~$\mathsf{ID}^\Gamma$, we mean a model in which~$\mathcal L_0\backslash\{Q\}$ is interpreted by the standard structure of natural numbers. The following proof is included for the convenience of the reader, even though it is a standard variant of known arguments (see~\cite{feferman-ID-Pi11CA}).

\begin{proposition}[$\mathsf{ACA}_0$]\label{prop:ID-to-Pi11}
Assume that $\mathcal M$ is an $\omega$-model of $\mathsf{ID}^\Gamma$. Let $\mathcal N\subseteq\mathcal P(\mathbb N)$ consist of the interpretations of all predicate symbols~$Q$ and $I_\varphi^n$ in~$\mathcal M$. If $Z$ is the interpretation of~$Q$, we have $\mathcal N\vDash\mathsf{ATR}_0+\Pi^1_1\text{-}\mathsf{CA}(Z)$.
\end{proposition}
\begin{proof}
To ensure $\mathcal N\vDash\mathsf{ATR}_0$ we only need $\mathcal N\vDash\mathsf{FP}$, by Theorem~\ref{thm:ATR-FP} (due to Avigad) and since arithmetical comprehension is a special case of~$\mathsf{FP}$. Consider an $X$-positive formula~$\varphi(x,X,Y)$ with parameter~$Y\in\mathcal N$. In order to validate the corresponding instance of~$(\mathsf{FP}$), we need to find an~$X\in\mathcal N$ with
\begin{equation*}
\mathcal N\vDash\forall x\big(x\in X\leftrightarrow\varphi(x,X,Y)\big).
\end{equation*}
We may assume that~$\varphi$ has no number parameters, as discussed above. The restric\-tion to a single set parameter is pure notational convenience. Assume that $Y$ is the interpretation of a predicate symbol~$I_\theta^n$ (the case of~$Q$ being analogous). We write $\psi$ for the positive $\mathcal L_{n+1}^X$-formula that results from~$\varphi(x,X,Y)$ when we replace all (positive and negative occurrences of) subformulas $t\in X$ and $t\in Y$ by $Xt$ and~$I_\theta^nt$, respectively. The corresponding instance of~($\mathsf F$)~yields
\begin{equation*}
\mathcal M\vDash\forall x\big(I_\psi^{n+1}x\leftrightarrow\psi[x,I_\psi^{n+1}]\big).
\end{equation*}
Given that $\varphi$ is arithmetical and hence absolute for $\omega$-models, our instance of ($\mathsf{FP}$) is thus satisfied when $X$ is the interpretation of~$I_\psi^{n+1}$ in~$\mathcal M$. 

It remains to show that
\begin{equation*}
\{x\in\mathbb N:\mathcal N\vDash\xi(x,Z)\}\in\mathcal N
\end{equation*}
holds for any $\Pi^1_1$-formula~$\xi$ with no set parameters other than~$Z$ (where the latter is the interpretation of~$Q$ in~$\mathcal M$). Since we already know $\mathcal N\vDash\mathsf{ACA}_0$, we get
\begin{equation*}
\mathcal N\vDash\xi(x,Z)\leftrightarrow\text{``$\vartriangleleft_x^Z$ is well-founded"}
\end{equation*}
for some family of binary relations~$\vartriangleleft_x^Z$ defined by a formula that is arithmetical (in fact~$\Delta^0_1$) with parameter~$Z$ (employ Lemma~V.1.4 of~\cite{simpson09} to write $\xi(x,Z)$ in the form $\forall f\!:\mathbb N\to\mathbb N\,\exists n\,\theta(x,f[n],Z)$ and declare that $\sigma\vartriangleleft_x^Z\tau$ holds for any sequences $\tau\sqsubset\sigma$ with $\neg\theta(x,\sigma[n],Z)$ for all~$n<|\sigma|$). Using Cantor pairing, we define $\rho(\langle x,y\rangle)$ as the positive $\mathcal L_0^X$-formula that results from $\forall y'(y'\vartriangleleft_x^Z y\to \langle x,y'\rangle\in X)$ when each sub\-formula~$t\in Z$ is replaced by $Qt$. Let $W\in\mathcal N$ be the interpretation of~$I_\rho^0$ in~$\mathcal M$. By the corresponding axioms ($\mathsf F$) and ($\mathsf L$) in~$\mathcal M$, we obtain
\begin{equation*}
W=\bigcap\left\{U\in\mathcal N:\mathcal N\vDash\forall y'(y'\vartriangleleft_x^Z y\to\langle x,y'\rangle\in U)\to\langle x,y\rangle\in U\!\right\}\!.
\end{equation*}
It follows that we have
\begin{equation*}
\mathcal N\vDash\forall y\,\langle x,y\rangle\in W\quad\Leftrightarrow\quad\mathcal N\vDash\text{``$\vartriangleleft_x^Z$ is well-founded"}.
\end{equation*}
To see this, first assume that the right side holds. For each $U$ from the intersection above, we then get $\langle x,y\rangle\in U$ by induction on~$y$ in the model~$\mathcal N$. This yields the left side. Now assume that the right side fails. We then have a non-empty set $V\in\mathcal N$ without a $\vartriangleleft_x^Z$-minimal element. Let $U\in\mathcal N$ consist of the pairs $\langle x',y\rangle$ with either $x'=x$ and $y\notin V$ or $x'\neq x$ (in which case~$y$ can be arbitrary). This $U$ contributes to the intersection that characterizes~$W$, so that we get $V\subseteq U$ and thus $\langle x,y\rangle\notin W$ for any~$y\in V$. We have shown
\begin{equation*}
\{x\in\mathbb N:\mathcal N\vDash\xi(x,Z)\}=\{x\in\mathbb N:\mathcal N\vDash\forall y\,\langle x,y\rangle\in W\}.
\end{equation*}
The right side reveals that our set lies in~$\mathcal N$, due to arithmetical comprehension.
\end{proof}

\section{\texorpdfstring{$\omega$}{omega}-completeness}\label{sec:omega-completeness}

In view of Proposition~\ref{prop:ID-to-Pi11} above, we want to show that there are $\omega$-models of~$\mathsf{ID}^\Gamma$. This can be achieved via a variant of the $\omega$-completeness theorem, which will be proved in the present section.

An $\mathcal L_\omega$-sequent (often just called sequent) is a finite set of $\mathcal L_\omega$-sentences. Let us note that free variables can be avoided in the infinitary proof system that is given by Definition~\ref{def:inf-proofs} below, where universal quantifiers are introduced by the $\omega$-rule (clause~(v) of the definition). A sequent~$\Delta$ should be interpreted as the disjunction of its elements. In the context of sequents, it is common to omit set parentheses and to use commata for unions (so that $\Delta,\varphi$ denotes~$\Delta\cup\{\varphi\}$).

Throughout this section, we fix an intended interpretation~$\mathcal Q\subseteq\mathbb N$ for the predicate symbol~$Q$ from~$\mathcal L_\omega$. For a closed term~$t$ with value~$n\in\mathbb N$, we agree that $\overline Qt$ denotes the formula~$Qt$ if we have $n\in\mathcal Q$ and the formula $\neg Qt$ otherwise. The following definition will only be invoked when the metatheory contains~$\mathsf{ATR}_0$, where arithmetical recursion on~$x\in X$ can be used to construct the sets $\{\Delta:{\vdash^x\Delta}\}$.

\begin{definition}\label{def:inf-proofs}
Consider a well-order~$X$. For~$x\in X$ and an $\mathcal L_\omega$-sequent~$\Delta$, we declare that $\vdash^x\Delta$ holds precisely if one of the following recursive clauses applies:
\begin{enumerate}[label=(\roman*)]
\item The sequent~$\Delta$ contains a true literal of first-order arithmetic, a formula~$\overline Qt$ or an instance of an axiom~($\mathsf F$) or~($\mathsf L$) (see the previous section).
\item The sequent~$\Delta$ contains~$I^n_\varphi s$ and $\neg I^n_\varphi t$ for~$s$ and $t$ with the same value.
\item The sequent~$\Delta$ contains a formula~$\varphi_0\land\varphi_1$ such that we have $\vdash^{x(0)}\Delta,\varphi_0$ and $\vdash^{x(1)}\Delta,\varphi_1$ for some~$x(0),x(1)<_X x$.
\item There is a formula~$\varphi_0\lor\varphi_1$ in~$\Delta$ and an index~$i\leq 1$ such that $\vdash^{y}\Delta,\varphi_i$ holds for some~$y<_X x$.
\item The sequent~$\Delta$ contains a formula~$\forall x\,\varphi(x)$ such that each~$n\in\mathbb N$ admits an~$x(n)<_X x$ with $\vdash^{x(n)}\Delta,\varphi(n)$ (where the last~$n$ denotes a numeral).
\item There is a formula $\exists x\,\varphi(x)$ in~$\Delta$ as well as a number~$n\in\mathbb N$ such that we have $\vdash^{y}\Delta,\varphi(n)$ for some~$y<_X x$.
\item There is an $\mathcal L_\omega$-sentence~$\varphi$ such that we have $\vdash^{x(0)}\Delta,\varphi$ and $\vdash^{x(1)}\Delta,\neg\varphi$ for some $x(0),x(1)<_X x$.
\end{enumerate}
\end{definition}

The following is a variant of $\omega$-completeness, which goes back to Sch\"utte~\cite{schuette56} and Shoenfield~\cite{shoenfield59}. As far as the authors are aware, no explicit reference for our variant is available, though similar results were at least implicitly proved by the same method (see, e.g., Section~4 of~\cite{rathjen-model-bi} for a variant over second-order arithmetic). Let us recall that we have fixed an interpretation~$\mathcal Q$ of the predicate symbol~$Q$.

\begin{proposition}[$\mathsf{ATR}_0$]
For any~$\mathcal L_\omega$-sequent~$\Delta$, the following are equivalent:
\begin{enumerate}[label=(\roman*)]
\item In every $\omega$-model of~$\mathsf{ID}^\Gamma$ that interprets~$Q$ by~$\mathcal Q$, some sentence from the sequent~$\Delta$ is satisfied.
\item We have $\vdash^x\Delta$ for some well-order~$X$ and some element~$x\in X$.
\end{enumerate}
\end{proposition}
\begin{proof}
    We first show that~(ii) implies~(i). The set of sentences that are true in a given $\omega$-model $\mathcal M\vDash\mathsf{ID}^\Gamma$ is defined by arithmetical recursion over~$\mathbb N$. In particular, it is available in $\mathsf{ATR}_0$. With this set as a parameter, we can use arithmetical induction on~$x\in X$ to show that $\vdash^x\Delta$ implies~$\mathcal M\vDash\varphi$ for some~$\varphi\in\Delta$, as the clauses from Definition~\ref{def:inf-proofs} are sound.

    To establish the converse implication, we recursively construct a tree~$T$ that is labelled by $\mathcal L_\omega$-sequents $l(\sigma)$ for all nodes~$\sigma\in T$. In the base case of the recursion, we declare that $T$ contains a root node~$\langle\rangle$ (the empty sequence), to which we assign the given sequent~$l(\langle\rangle)=\Delta$. For the recursion step, assume that we have just added a finite sequence~$\sigma\in T$ with label $l(\sigma)$. If $l(\sigma)$ validates clause~(i) or~(ii) from Definition~\ref{def:inf-proofs} (meaning that one of these clauses holds with $l(\sigma)$ at the place of~$\Delta$), then we declare that~$\sigma$ is a leaf of~$T$, i.e., we do not add extensions of~$\sigma$. Otherwise, we distinguish two cases depending on the length of our sequence.
    
    First assume that~$\sigma$ has even length. We fix an enumeration of all formulas~$\alpha_n$ that are instances of the axiom schemata $(\mathsf F)$ and $(\mathsf L)$ or axioms of the form~$\overline Qt$. By $\sigma\star i$ we denote the extension of~$\sigma$ by a new last element~$i$. If the length of~$\sigma$ is equal to~$2n$, we now declare that we have $\sigma\star i\in T$ for both $i\in\{0,1\}$, where the labels are given by
    \begin{equation*}
        l(\sigma\star 0)=l(\sigma),\alpha_n\quad\text{and}\quad l(\sigma\star 1)=l(\sigma),\neg\alpha_n.
    \end{equation*}
    This corresponds to an application of the cut rule (clause~(vii) of Definition~\ref{def:inf-proofs}). We stipulate that $\sigma\star 0$ is always a leaf in this situation.

    Now assume that $\sigma$ has odd length and the form~$\sigma=\tau\star 1$. The idea is to add labelled descendants of~$\sigma$ by applying the rules~(iii) to~(vi) of Definition~\ref{def:inf-proofs} backwards in all possible ways. Let us write $l(\sigma)=\varphi_\sigma,\Gamma_\sigma$ to single out the first sentence in our sequent (which is non-empty be the even case), according to some fixed enumeration. We include the following descendants:
    \begin{itemize}
        \item If $\varphi_\sigma$ is a literal, let $\sigma\star 0\in T$ and $l(\sigma\star i)=\Gamma_\sigma, \varphi_\sigma$.
        \item If $\varphi_\sigma=\psi_0\wedge \psi_1$, let $\sigma\star i\in T$ and  $l(\sigma\star i)=\Gamma_\sigma,\varphi_\sigma,\psi_i$ for each $i\in\{0,1\}$.
        \item If $\varphi_\sigma=\psi_0\vee \psi_1$, let $\sigma\star 0\in T$ and  $l(\sigma\star 0)=\Gamma_\sigma,\varphi_\sigma,\psi_j$ with
        \begin{equation*}
            j=\begin{cases}
                0 & \text{if }\psi_0\notin\Gamma_\sigma,\\
                1 & \text{otherwise}.
            \end{cases}
        \end{equation*}
        \item If $\varphi_\sigma=\forall x\,\psi(x)$, let $\sigma \star n\in T$ and $l(\sigma\star n)= \Gamma_\sigma, \varphi_\sigma,\psi(n)$ for all $n\in\mathbb N$.
        \item If $\varphi_\sigma=\exists x\,\psi(x)$, let $\sigma \star 0\in T$ and $l(\sigma\star 0)=\Gamma_\sigma,\varphi_\sigma,\psi(n)$ with
        \begin{equation*}
            n=\min\{m\in\mathbb N:\psi(m) \not\in\Gamma_\sigma\}.
        \end{equation*}
    \end{itemize}
    Note that each of these clauses permutes the sentences in the sequent $l(\sigma)$. This guarantees that each sentence in the sequent is examined eventually, unless a leaf is found before.

    Now that the construction of our labelled tree is complete, we distinguish two outcomes. First assume that~$T$ has no branch, i.e., that there is no~$f:\mathbb N\to\mathbb N$ such that $f[n]=\langle f(0),\ldots,f(n-1)\rangle\in T$ holds for all~$n\in\mathbb N$. Then the Kleene-Brouwer order is a well-order on~$T$. We denote this order by~$X$ and note that we always have $\sigma\star i<_X\sigma$. A straightforward induction on~$\sigma\in X$ yields~$\vdash^\sigma l(\sigma)$ and in particular~$\vdash^{\langle\rangle}\Delta$, so that statement~(ii) from the proposition is satisfied.

    Finally, assume that~$T$ has a branch~$f$. To conclude that the implication from~(i) to~(ii) is valid once again, we show that~(i) is false in this case. Let $\mathcal F=\bigcup_{n\in\mathbb N}l(f[n])$ be the set of formulas that occur on our branch. To define an $\omega$-model~$\mathcal M\vDash\mathsf{ID}^\Gamma$, we stipulate
    \begin{equation*}
        P^\mathcal M=\{n\in\mathbb N:Ps\notin\mathcal F\text{ for every term~$s$ with value~$n$}\},
    \end{equation*}
    where~$P$ ranges over the predicate symbols~$Q$ and~$I^n_\varphi$. Induction over the build-up of formulas shows that $\varphi\in\mathcal F$ entails~$\mathcal M\nvDash\varphi$. We present two representative cases of the inductive proof.
    
    First assume that~$\varphi\in\mathcal F$ has the form~$\neg Pt$, where $t$ is a term with value~$n\in\mathbb N$. When $s$ is any term with value~$n$, we have~$Ps\notin\mathcal F$. Otherwise, since $l(\sigma)$ grows whenever~$\sigma$ is extended, we would find a single~$k\in\mathbb N$ such that $l(f[k])$ contains both $Ps$ and~$\neg Pt$. But then $f[k]$ would be a leaf of~$T$ (cf.~clauses~(i,ii) of Definition~\ref{def:inf-proofs}), against the assumption that~$f$ is a branch. So we have $n\in P^{\mathcal M}$ and thus $\mathcal M\nvDash\varphi$.
    
    As a second case, consider a sentence~$\varphi\in\mathcal F$ of the form~$\exists x\,\psi(x)$. To get $\mathcal M\nvDash\varphi$, it is enough to show that we have $\psi(n)\in\mathcal F$ and hence inductively $\mathcal M\nvDash\psi(n)$ for all~$n\in\mathbb N$. Towards a contradiction, assume that $n$ is minimal with~$\psi(n)\notin\mathcal F$. We find a single~$k$ such that $l(f[k])$ contains $\varphi$ and $\psi(m)$ for all~$m<n$. Rotating our sequent, we increase to an odd~$k$ such that $\varphi$ is the first formula of~$l(f[k])$. By the construction of~$T$, we get $\psi(n)\in l(f[k+1])\subseteq\mathcal F$, against our assumption.

    For each axiom $\alpha_n$ from the enumeration that was fixed above, the construction at even stages yields $\neg\alpha_n\in l(f[2n+1])\subseteq\mathcal F$ and hence~$\mathcal M\vDash\alpha_n$. This shows that we have $\mathcal M\vDash\mathsf{ID}^\Gamma$ as well as $\mathcal M\vDash\overline Q m$ for all~$m\in\mathbb N$. The latter yields~$\mathcal Q=Q^{\mathcal M}\in\mathcal M$. Finally, we have $\mathcal M\nvDash\varphi$ for every~$\varphi\in\Delta=l(f[0])\subseteq\mathcal F$. So statement~(i) from the proposition is false in the present case, as desired.
\end{proof}

We will use the proposition in the following form, where the sequent~$\emptyset$ (which corresponds to the empty disjunction) is a canonical representation of falsity.

\begin{corollary}[$\mathsf{ATR}_0$]\label{cor:omega-completeness}
The following are equivalent:
\begin{enumerate}[label=(\roman*)]
\item There is an $\omega$-model of~$\mathsf{ID}^\Gamma$ that interprets~$Q$ as~$\mathcal Q$.
\item We have $\nvdash^x\emptyset$ for every well-order~$X$ and every~$x\in X$.
\end{enumerate}
\end{corollary}
\begin{proof}
For $\Delta=\emptyset$, statements~(i) and~(ii) from the proposition are the negations of the corresponding statements from the corollary.
\end{proof}

Concerning the equivalence in our main Theorem~\ref{thm:main}, the direction from~(i) to~(ii) has been established in Section~\ref{sect:well-ordering-proof}. For the converse direction, we note that statement~(i) of the theorem reduces -- by means of Proposition~\ref{prop:ID-to-Pi11} -- to the equivalent statements from the corollary above.

\section{Ordinal analysis: the embedding}

In order to complete the proof of our main Theorem~\ref{thm:main}, we will show that the well-ordering principle from part~(ii) of the theorem implies statement~(ii) from Corollary~\ref{cor:omega-completeness}. The latter is a consistency statement and hence a classical target for an ordinal analysis, which we carry out in this section and the two following ones.

\begin{assumption}\label{as:Gamma-Omega-X}
    From this point on, we fix a well-order~$Y=1+X$ and assume that the associated order~$\Gamma_{\Omega+Y}$ (as constructed in Section~\ref{sect:ordinal-notations}) is also a well-order. As in the previous section, we also fix an interpretation~$\mathcal Q\subseteq\mathbb N$ for the predicate symbol~$Q$ from~$\mathcal L_\omega$. These assumptions will be discharged in the last proof of this paper.
\end{assumption}

The point of considering $Y=1+X$ rather than~$X$ is that $\Omega+0\in\Omega+Y$ yields an element $\Gamma_{\Omega+0}\in\Gamma_{\Omega+Y}$, which we also denote by $\Omega$ (as we have $\Gamma_\Omega=\Omega$ under classical interpretations of~$\Omega$ as, say, the least uncountable cardinal). In view of Definition~\ref{def:psi}, we get
\begin{equation*}
    \psi(\Gamma_{\Omega+Y})=\{s\in\Gamma_{\Omega+Y}:s\prec\Omega\}.
\end{equation*}
Let $\mathcal L_\omega^\Omega$ be the language that results from $\mathcal L_\omega$ when each predicate symbol~$I^0_\varphi$ is replaced by predicate symbols $I^0_{\varphi,\alpha}$ for all~$\alpha\in\Gamma_{\Omega+Y}$ with $\alpha\preceq\Omega$. We view $\mathcal L_\omega$ as a sublanguage of $\mathcal L_\omega^\Omega$ by identifying $I^0_\varphi$ with $I^0_{\varphi,\Omega}$. Note that the predicates~$I^n_\psi$ with $n>0$ are unaffected by the modification of the language. To explain why the higher levels are treated differently from the lowest one, we recall that axiom~($\mathsf L$) is only available at the latter. Intuitively, the predicates~$I^0_{\varphi,\alpha}$ serve as approximations to the fixed point~$I^0_\varphi$, which is reached at level~$\alpha=\Omega$. We will use these approximations to prove axiom~($\mathsf L$) in a new infinitary proof system. The latter will then embed the proof system from Definition~\ref{def:inf-proofs}. Whenever we refer to Definition~\ref{def:inf-proofs}, it is with respect to the order~$X$ (not~$Y$) that was fixed in Assumption~\ref{as:Gamma-Omega-X}.

The following relies on basic ordinal arithmetic as intro\-duced in Definition~\ref{def:Gamma-operations}. 

\begin{definition}\label{def:ranks-big-system}
    The rank~$|\psi|\in\Gamma_{\Omega+Y}$ of an $\mathcal L_\omega^\Omega$-formula~$\psi$ is recursively given by
    \begin{gather*}
        |\theta|=0\quad\text{when $\theta$ is a literal of }\mathcal L_0,\\
        |I^0_{\varphi,\alpha}t|=|\neg I^0_{\varphi,\alpha}t|=\omega\cdot(1+\alpha),\\
        |I^n_\varphi t|=|\neg I^n_\varphi t|=\Omega+\omega\cdot n\quad\text{when }n>0,\\
        |\psi_0\land\psi_1|=|\psi_0\lor\psi_1|=\max(|\psi_0|,|\psi_1|)+1,\\
        |\forall x\,\psi|=|\exists x\,\psi|=|\psi|+1.
    \end{gather*}
\end{definition}

To reduce the number of rules in our proof system (and hence the number of cases that we need to consider), we associate some formulas with generalized disjunctions and conjunctions (cf.~\cite{buchholz-local-predicativity}). The general approach is to write some formulas~$\psi$ as
\begin{equation*}
\psi\simeq\textstyle\bigvee_{i\prec\iota(\psi)}\psi_i.
\end{equation*}
This means that we assign a type (`disjunctive'), an arity~$\iota(\psi)\in\Gamma_{\Omega+Y}$ as well as a formula~$\psi_i$ for each~$i\prec\iota(\psi)$. In case we have~$\iota(\varphi)=0$ (so that there are no formulas~$\varphi_i$), we speak of an empty disjunction, which indicates a false formula. When the formulas of disjunctive type are given, we define the conjunctive formulas by
\begin{equation*}
    \neg\psi\simeq\textstyle\bigwedge_{i\prec\iota(\neg\psi)}\neg\psi_i\quad\text{for}\quad\psi\simeq\textstyle\bigvee_{i\prec\iota(\psi)}\psi_i,
\end{equation*}
which amounts to $\iota(\neg\psi)=\iota(\psi)$ and $(\neg\psi)_i=\neg(\psi_i)$. The last two equations also hold when $\psi$ is conjunctive and $\neg\psi$ is disjunctive. Indeed, if we write our disjunctive formula $\neg\psi$ as~$\psi'$, then $\neg\psi'$ coincides with the given conjunctive formula~$\psi$ (as we always write formulas in negation normal form and compute negations via de Morgan's rules and double negation elimination -- which also yields $|\neg\psi|=|\psi|$). We will never declare both a formula and its negation as disjunctive, so that any formula will have at most one type (disjunctive or conjunctive), though some formulas may have none. In the case at hand, this general approach is implemented as follows.

\begin{definition}\label{def:disj}
    Each false literal of first-order arithmetic and each formula~$\neg\overline Qt$ is associated with an empty disjunction (see the previous paragraph). Other than that, the $\mathcal L_\omega^\Omega$-sentences of disjunctive type are given by
    \begin{equation*}
        \psi_0\lor\psi_1\simeq\textstyle\bigvee_{i\prec2}\psi_i,\quad\exists x\,\psi(x)\simeq\bigvee_{i\prec\omega}\psi(i),\quad I^0_{\varphi,\alpha}t\simeq\bigvee_{\beta\prec\alpha}\varphi[t,I^0_{\varphi,\beta}],
    \end{equation*}
    where the $i$ in $\psi(i)$ refers to a numeral.
\end{definition}

The following is crucial for cut elimination in the presence of generalized disjunctions and conjunctions.

\begin{lemma}[$\mathsf{ATR}_0$]\label{lem:rank-subformulas}
    We have $|\psi_i|\prec|\psi|$ for $i\prec\iota(\psi)$.
\end{lemma}
\begin{proof}
    In the interesting case, $\psi$ is of the form~$I^0_{\varphi,\alpha}t$. Here $\varphi$ is an $\mathcal L_0^X$-formula, which means that it contains no predicates~$I^0_{\theta,\gamma}$ or $I^n_\theta$. For $\beta\prec\iota(\varphi)=\alpha$, we thus get $|\varphi[t,I^0_{\varphi,\beta}]|\prec\omega\cdot(1+\beta+1)\preceq\omega\cdot(1+\alpha)=|\psi|$.
\end{proof}

As the previous proof suggests, the rank of a formula is largely determined by the following coefficients (traditionally denoted $k(\psi)$ for the German `Koeffizienten').

\begin{definition}\label{def:formula-coeffs}
For an $\mathcal L_\omega^\Omega$-sentence~$\psi$, we put
\begin{equation*}
    k(\psi)=\{\alpha:\text{a predicate symbol }I^0_{\varphi,\alpha}\text{ occurs in $\psi$}\}.
\end{equation*}
When $\Delta$ is an $\mathcal L_\omega^\Omega$-sequent, we set $k(\Delta)=\bigcup_{\psi\in\Delta}k(\psi)$.
\end{definition}

In the following proof system, the coefficients and ordinal heights are controlled by the operators that were introduced at the end of Section~\ref{sect:ordinal-notations}. This operator control (pioneered by Buchholz~\cite{buchholz-local-predicativity}) will only become relevant in the collapsing procedure, which we present in the final section of this paper. Until then, it appears like~an incidental feature of the proof system, which must nevertheless be tracked through\-out the development.

\begin{definition}\label{def:H-proofs}
    For an operator~$\mathcal H$ and $\alpha,\rho\in\Gamma_{\Omega+Y}$ as well as an $\mathcal L^\Omega_\omega$-sequent~$\Delta$, we declare that $\mathcal H\vdash^\alpha_\rho\Delta$ holds when we have $\{\alpha\}\cup k(\Delta)\subseteq\mathcal H(\emptyset)$ and one of the following recursive clauses applies:
    \begin{enumerate}[label=(\roman*)]
    \item[(\textsf{Ax})] The sequent~$\Delta$ contains formulas $I^n_\varphi s$ and $\neg I^n_\varphi t$ for terms~$s$ and $t$ with the same value (where $n>0$).
    \item[$(\bigwedge)$] There is a conjunctive $\varphi\simeq \bigwedge_{i\prec\iota(\varphi)}\varphi_i\in \Delta$ such that $\mathcal H[i]\vdash^{\alpha(i)}_\rho\Delta,\varphi_i$ holds for all~$i\prec\iota(\varphi)$ with some~$\alpha(i)\prec\alpha$.
    \item[$(\bigvee)$] There is a disjunctive $\varphi\simeq \bigvee_{i\prec\iota(\varphi)}\varphi_i\in\Delta$ such that $\mathcal H\vdash^{\beta}_\rho\Delta, \varphi_i$ holds for some~$i\prec\min(\iota(\varphi),\alpha)$ with $i\in\mathcal H(\emptyset)$ and some $\beta\prec\alpha$.
    \item[(\textsf{Cut})] There is an $\mathcal L_\omega^\Omega$-sentence~$\psi$ with $|\psi|\prec\rho$ such that we have  $\mathcal H\vdash^{\alpha(0)}_\rho\Delta,\psi$ and $\mathcal H\vdash^{\alpha(1)}_\rho\Delta,\neg\psi$ for some $\alpha(0),\alpha(1)\prec\alpha$.
    \item[($\mathsf F^\Omega$)] We have $\Omega\preceq\alpha$ and there is $I^0_{\varphi,\Omega}t\in\Delta$ with $\mathcal H\vdash^{\beta}_\rho\Delta,\varphi[t,I^0_{\varphi,\Omega}]$ for~$\beta\prec\alpha$.
    \item[($\mathsf F^+$)] There is $I^n_\varphi\, t\in\Delta$ with $\mathcal H\vdash^{\beta}_\rho\Delta,\varphi[t,I^n_\varphi]$ for some~$\beta\prec\alpha$ (where $n>0$).
    \item[($\mathsf F^-$)] There is $\neg I^n_\varphi\, t\in\Delta$ with $\mathcal H\vdash^{\beta}_\rho\Delta,\neg\varphi[t,I^n_\varphi]$ for~$\beta\prec\alpha$ (where again $n>0$).
\end{enumerate}
\end{definition}

To see how the definition is formalized in $\mathsf{ATR}_0$, recall that all our operators are of the form~$\mathcal H_\gamma[a]$ with $\{\gamma\}\cup a\in\Gamma_{\Omega+Y}$, where~$a$ is finite (see Definition~\ref{def:H-operator}). The previous definition uses arithmetical recursion on~$\alpha$ to construct the family of sets
\begin{equation*}
    \left\{(\gamma,a,\rho,\Delta):\mathcal H_\gamma[a]\vdash^\alpha_\rho\Delta\right\}.
\end{equation*}
We also note that the induction statement in the following proof is arithmetical (as quantification over operators~$\mathcal H=\mathcal H_\gamma[a]$ amounts to quantification over~$\gamma$ and~$a$). The result is standard but requires some attention under operator control.

\begin{lemma}[$\mathsf{ATR}_0$; `Weakening']\label{lem:weakening}
    The implication
    \begin{equation*}
        \mathcal H\vdash^\alpha_\rho\Delta\quad\Rightarrow\quad\mathcal H'\vdash^\beta_\pi\Delta,\Delta'
    \end{equation*}
    holds when we have $\mathcal H(a)\subseteq\mathcal H'(a)$ for all finite $a\subseteq\Gamma_{\Omega+Y}$ as well as $\alpha\preceq\beta\in\mathcal H'(\emptyset)$ and $\rho\preceq\pi$ and $k(\Delta')\subseteq\mathcal H'(\emptyset)$.
\end{lemma}
\begin{proof}
    We argue by induction on~$\alpha$ and distinguish cases according to the previous definition. In the most interesting case, we have a conjunctive $\varphi\simeq \bigwedge_{i\prec\iota(\varphi)}\varphi_i\in \Delta$ such that
    \begin{equation*}
        \mathcal H[i]\vdash^{\alpha(i)}_\rho\Delta,\varphi_i\quad\text{with}\quad\alpha(i)\prec\alpha
    \end{equation*}
    holds for all~$i\prec\iota(\varphi)$. Note that we get $\alpha(i)\in\mathcal H[i](\emptyset)$ by the initial condition from Definition~\ref{def:H-proofs}. Any finite $a\subseteq\Gamma_{\Omega+Y}$ validates
    \begin{equation*}
        \mathcal H[i](a)=\mathcal H(\{i\}\cup a)\subseteq\mathcal H'(\{i\}\cup a)=\mathcal H'[i](a).
    \end{equation*}
    Inductively, we obtain
    \begin{equation*}
        \mathcal H'[i]\vdash^{\alpha(i)}_\pi\Delta,\Delta',\varphi_i
    \end{equation*}
    for every~$i\prec\iota(\varphi)$. To infer $\mathcal H'\vdash^\beta_\pi\Delta,\Delta'$ by the rule ($\bigwedge$), it suffices to note that the condition $\{\beta\}\cup k(\Delta,\Delta')\subseteq\mathcal H'(\emptyset)$ is satisfied by assumption.
\end{proof}

Our next result is also standard but somewhat subtle in the presence of operators.

\begin{lemma}[$\mathsf{ATR}_0$]\label{lem:psi-neg-psi}
    For any operator~$\mathcal H$ and any $\mathcal L_\omega^\Omega$-sentence $\psi$, we have
    \begin{equation*}
        \mathcal H[k(\psi)]\vdash^{\omega\cdot|\psi|}_0\psi,\neg\psi.
    \end{equation*}
\end{lemma}
\begin{proof}
    We have $k(\psi)\subseteq\mathcal H(k(\psi))=\mathcal H[k(\psi)](\emptyset)$ by Lemma~\ref{lem:closure-op}. Also, $|\psi|$ has one of the forms $n$ and $\Omega+\omega\cdot n$ and $\omega\cdot(1+\alpha)+n$ with $n\prec\omega$ and $\alpha\in k(\psi)$. According to Remarks~\ref{rmk:omega*alpha-closure} and~\ref{rmk:H-closure}, we get $\omega\cdot|\psi|+1\in\mathcal H[k(\psi)](\emptyset)$, so that the initial condition from Definition~\ref{def:H-proofs} is satisfied. It remains to check that one of the rules from this definition can be applied. First assume that neither $\psi$ nor $\neg\psi$ is disjunctive in the sense of Definition~\ref{def:disj}. Then $\psi$ or $\neg\psi$ is of the form $I^n_\varphi t$ with $n>0$. Here we can conclude by ($\mathsf{Ax}$). Now assume that one of the two formulas is disjunctive, say with
    \begin{equation*}
        \psi\simeq\textstyle\bigvee_{i\prec\iota(\varphi)}\psi_i\quad\text{and}\quad\neg\psi\simeq\textstyle\bigwedge_{i\prec\iota(\varphi)}\neg\psi_i.
    \end{equation*}
    For each~$i\prec\iota(\psi)$, we have $k(\psi_i)\subseteq k(\psi)\cup\{i\}$. So the induction hypothesis and weakening yield
    \begin{equation*}
        \mathcal H[k(\psi)][i]\vdash^{\omega\cdot|\psi_i|}_0\psi,\neg\psi,\psi_i,\neg\psi_i.
    \end{equation*}
    We can apply rule ($\bigvee$) to get
    \begin{equation*}
        \mathcal H[k(\psi)][i]\vdash^\alpha_0\psi,\neg\psi,\neg\psi_i\quad\text{with}\quad\alpha_i=\max(\omega\cdot|\psi_i|,i)+1.
    \end{equation*}
    In view of $i\prec\iota(\psi)\preceq\omega\cdot|\psi|$ we get $\alpha_i\prec\omega\cdot|\psi|$. Thus the claim follows by an application of~($\bigwedge$).
\end{proof}

In the rest of this section, we show that the proof system from Definition~\ref{def:inf-proofs} (see also Corollary~\ref{cor:omega-completeness} and the discussion that follows it) embeds into the system from Definition~\ref{def:H-proofs}. By the following result, the latter proves all instances of axiom~($\mathsf F$). Concerning the case of~$n=0$, we recall that $I^0_\varphi$ is identified with~$I^0_{\varphi,\Omega}$.

\begin{proposition}[$\mathsf{ATR}_0$]\label{prop:proof-F}
    For any operator~$\mathcal H$ and predicate symbol~$I^n_\varphi$, we have
    \begin{equation*}
        \mathcal H\vdash^{\Omega+\omega^3}_0\forall x\big(\varphi[x,I^n_\varphi]\to I^n_\varphi x\big).
    \end{equation*}
    When we have $n>0$, we also get $\mathcal H\vdash^{\Omega+\omega^3}_0\forall x(I^n_\varphi x\to\varphi[x,I^n_\varphi])$.
\end{proposition}
\begin{proof}
We have $|\varphi[m,I^n_\varphi]|<\Omega+\omega\cdot(n+1)$ and $k(\varphi[m,I^n_\varphi])\subseteq\{\Omega\}$. The latter holds because $\varphi$ is an $\mathcal{L}_\omega$-formula (rather than an $\mathcal L_\omega^\Omega$-formula), which can contain predicates $I^0_{\theta,\Omega}$ as identified with $I^0_\theta$ but no predicates $I^0_{\theta,\eta}$ for $\eta\prec\Omega$. Given that we have $\mathcal H[\Omega]=\mathcal H$ by Lemma~\ref{lem:C-sets-closure}(d) and Corollary~\ref{cor:H[b]=H} (recall~$\Omega=\Gamma_{\Omega+0}$), the previous lemma and weakening yield
\begin{equation*}
\mathcal H\vdash^{\Omega+\omega^2\cdot(n+1)}_0\neg\varphi[m,I^n_\varphi],\varphi[m,I^n_\varphi]
\end{equation*}
for every~$m\prec\omega$. By ($\mathsf F^\Omega$) or ($\mathsf F^+$) when $n=0$ or $n>0$, respectively, we get
    \begin{equation*}
        \mathcal H\vdash^{\Omega+\omega^2\cdot(n+2)}_0\neg\varphi[m,I^n_\varphi],I^n_\varphi m.
    \end{equation*}
In view of $\neg\psi_0\to\psi_1\simeq\bigvee_{i\prec 2}\psi_i$ and $\forall x\,\psi(x)\simeq\bigwedge_{m\prec\omega}\psi(m)$, we can conclude by clauses~($\bigvee$) and~($\bigwedge$). For $n>0$, we use ($\mathsf F^-$) to prove the converse implication.
\end{proof}

In Definition~\ref{def:H-proofs}, the only aspect of terms that we consider is the value. This has the following consequence.

\begin{remark}\label{rmk:term-value-variant}
    Our infinitary proof system allows us to replace any term by another term with the same value. More precisely, when $\Delta(x_1,\ldots,x_n)$ is a sequent with the indicated free variables, we have
    \begin{equation*}
        \mathcal H\vdash^\alpha_\rho\Delta(s_1,\ldots,s_n)\quad\Rightarrow\quad\mathcal H\vdash^\alpha_\rho\Delta(t_1,\ldots,t_n)
    \end{equation*}
    whenever each pair~$(s_i,t_i)$ consists of closed terms with equal value (in the standard model). One can check this by a straightforward induction on~$\alpha$.
\end{remark}

The following will be crucial to prove instances of axiom~($\mathsf L$).

\begin{lemma}[$\mathsf{ATR}_0$]
 Consider $\mathcal L_\omega^\Omega$-formulas $\theta(y)$ and $\psi(y)$ such that
 \begin{equation*}
     \mathcal H\vdash^\alpha_\rho\Delta,\neg\theta(m),\psi(m)
 \end{equation*}
 with $\alpha\succeq\omega$ holds for every $m\prec\omega$. If $\varphi$ is a positive $\mathcal L_0^X$-formula with a single free number variable, there is an~$h\prec\omega$ such that all $n\prec\omega$ validate
 \begin{equation*}
     \mathcal H\vdash^{\alpha+2h}_\rho\Delta,\neg\varphi[n,\lambda y\,\theta],\varphi[n,\lambda y\,\psi].
 \end{equation*}
\end{lemma}
\begin{proof}
    We argue by induction on the height~$h$ of the formula~$\varphi$. This requires us to treat the more general case where $\varphi(x_1,\ldots,x_k)$ has several free variables. Generalizing our previous notation in the obvious way, we then write $\varphi[n_1,\ldots,n_k,\lambda y\,\psi]$ or shorter $\varphi[\overline n,\lambda y\,\psi]$ (where~$\overline n$ is a tuple of numerals).

    First assume that $\varphi(\overline x)$ is of the form~$Xt(\overline x)$ for a term~$t$. In view of the previous remark, the assumption (with the value of~$t(\overline n)$ as $m$) gives
    \begin{equation*}
     \mathcal H\vdash^\alpha_\rho\Delta,\neg\theta(t(\overline n)),\psi(t(\overline n))
 \end{equation*}
 for any tuple~$\overline n$ of numbers. This is what we need, since $\varphi[\overline n,\lambda y\,\psi]$ is the same formula as~$\psi(t(\overline n))$. When~$\varphi$ is a (negated) atom of~$\mathcal L_0$, the desired conclusion reads~$\mathcal H\vdash^\alpha_\rho\Delta,\neg Pt(\overline n),Pt(\overline n)$ for a predicate~$P$ other than~$X$. It holds by clause~($\mathsf{Ax}$) or with a slightly worse bound by Lemma~\ref{lem:psi-neg-psi} (where $k(\Delta)\subseteq\mathcal H(\emptyset)$ is ensured by the assumption $\mathcal H\vdash^\alpha_\rho\Delta,\neg\theta(m),\psi(m)$).

 In the induction step, we consider one representative case, where~$\varphi(\overline x)$ is of the form~$\exists z\,\varphi_0(\overline x,z)$. Consider arbitrary~$\overline n$ and note $\varphi[\overline n,\lambda_y\,\psi]\simeq\bigvee_{k\prec\omega}\varphi_0[\overline n,k,\lambda_y\,\psi]$. For any~$k$, the induction hypothesis provides
 \begin{equation*}
\mathcal H\vdash^{\alpha+2h-2}_\rho\Delta,\neg\varphi_0[\overline n,k,\lambda y\,\theta],\varphi_0[\overline n,k,\lambda y\,\psi].
 \end{equation*}
 The assumption~$\alpha\succeq\omega$ in the lemma ensures $k\prec\alpha+2h-2$, which is needed to apply clause~($\bigvee$) from Definition~\ref{def:H-proofs}. By the latter, we get
 \begin{equation*}
\mathcal H\vdash^{\alpha+2h-1}_\rho\Delta,\neg\varphi_0[\overline n,k,\lambda y\,\theta],\varphi[\overline n,\lambda y\,\psi].
 \end{equation*}
An application of ($\bigwedge$) completes the induction step.
\end{proof}

As promised, we now obtain infinitary proofs for axiom~($\mathsf L$).

\begin{proposition}[$\mathsf{ATR}_0$]\label{prop:proof-L}
    For any predicate~$I_\varphi^0$ and any $\mathcal L_\omega$-formula~$\psi$, we have
    \begin{equation*}
        \mathcal H\vdash_0^{\Omega\cdot 2+\omega}\forall\overline z\bigg(\forall x\big(\varphi[x,\lambda x\,\psi]\to\psi(x,\overline z)\big)\to\forall x\big(I^0_{\varphi,\Omega}x\to\psi(x,\overline z)\big)\bigg).
    \end{equation*}
\end{proposition}

\begin{proof}
We fix values~$\overline m$ for the variables~$\overline z$ and mostly leave them implicit, i.e., we write $\psi$ for $\psi(x,\overline m)$. With this in mind, we write $\text{Cl}_\varphi(\psi)$ for $\forall x\big(\varphi[x,\lambda x\,\psi]\to\psi\big)$. Let us set $\beta=|\psi(n)|+1$, which is independent of~$n$ and also of~$\overline m$. By in\-duc\-tion over~$\delta\preceq\Omega$, we will show that all $n\prec\omega$ validate
    \begin{equation*}
        \mathcal H[\delta]\vdash_0^{\omega\cdot(\beta +\delta)}\neg \text{Cl}_\varphi(\psi), \neg I^0_{\varphi,\delta}\,n, \psi(n).
    \end{equation*}
For $\delta=\Omega$ this will yield the claim: We have $\mathcal{H}[\Omega] = \mathcal{H}$ and always $|\psi(n)|\prec\Omega+\omega^2$, so that we get $\omega\cdot(\beta+\Omega)\preceq\Omega\cdot 2$ and hence
    \begin{equation*}
        \mathcal H\vdash_0^{\Omega\cdot 2}\neg\forall x\big(\varphi[x,\lambda x\,\psi(x,\overline m)]\to\psi(x,\overline m)\big), \neg I^0_{\varphi,\Omega}\,n, \psi(n,\overline m).
    \end{equation*}
We conclude by applications of ($\bigwedge$) and ($\bigvee$).

In the remaining argument by induction, first observe $k(\psi)\subseteq\{\Omega\}\subseteq\mathcal{H}(\emptyset)$. By the induction hypothesis and previous proposition, each~$\gamma\prec\delta$ admits an $h\prec\omega$ such that all $n\prec\omega$ validate
\begin{equation*}
\mathcal H[\gamma]\vdash^{\omega\cdot(\beta+\gamma)+2h}_0\neg \text{Cl}_\varphi(\psi), \neg\varphi[n,I^0_{\varphi,\gamma}],\varphi[n,\lambda x\,\psi].
\end{equation*}
As Lemma~\ref{lem:psi-neg-psi} provides $\mathcal H\vdash^{\omega\cdot\beta}_0\psi(n),\neg\psi(n)$, we can infer
\begin{equation*}
\mathcal H[\gamma]\vdash^{\omega\cdot(\beta+\gamma)+2h+1}_0\neg \text{Cl}_\varphi(\psi), \neg\varphi[n,I^0_{\varphi,\gamma}],\varphi[n,\lambda x\,\psi]\land\neg\psi(n),\psi(n).
\end{equation*}
Now observe that we have
\begin{equation*}
\neg \text{Cl}_\varphi(\psi) \simeq\textstyle\bigvee_{n\prec\omega}\varphi[n,\lambda x\, \psi]\wedge\neg\psi(n). 
\end{equation*}
We can thus apply~($\bigvee$) to get
\begin{equation*}
\mathcal H[\gamma]\vdash^{\omega\cdot(\beta+\gamma)+2h+2}_0\neg \text{Cl}_\varphi(\psi), \neg\varphi[n,I^0_{\varphi,\gamma}],\psi(n).
\end{equation*}
By weakening, we can replace~$\mathcal H[\gamma]$ with~$\mathcal H[\delta][\gamma]$. In view of $\neg I^0_{\varphi,\delta}\simeq\bigwedge_{\gamma\prec\delta}\neg\varphi[n,I^0_{\varphi,\gamma}]$ and $k(I^0_{\varphi,\delta}n)=\{\delta\}\subseteq\mathcal H[\delta](\emptyset)$, we can apply~($\bigwedge$) to complete the induction step.   
\end{proof}

We state the following embedding result in the form of a remark and without detailed proof, as an improved embedding will be obtained in the following section.

\begin{remark}
    For any $\mathcal L_\omega$-sequent~$\Delta$ and every operator~$\mathcal H$, we have
    \begin{equation*}
        \vdash^x\Delta\quad\Rightarrow\quad\mathcal H\vdash^{\Gamma_{\Omega+1+x}}_{\Omega+\omega^2}\Delta.
    \end{equation*}
    This can be shown by induction over~$x$ in our fixed well-order~$X$. In the induction step, one distinguishes cases according to the clauses from Definition~\ref{def:inf-proofs}. The axioms~($\mathsf F$) and~($\mathsf L$) -- which appear in clause~(i) -- are covered by the previous results in the present section.
\end{remark}

\section{Ordinal analysis: the predicative part}\label{sect:pred-ord-ana}

For $N>0$ we define $\mathcal L_N^\Omega$ as the sublanguage of $\mathcal L_\omega^\Omega$ that omits all predicates~$I^n_\varphi$ with $n\geq N$. In this section, we eliminate detours through~$\mathcal L_{N+1}^\Omega$-formulas -- and corresponding applications of the rules ($\mathsf F^+$) and~($\mathsf F^-$) -- from proofs of \mbox{$\mathcal L_N^\Omega$-}formulas (in the proof system from Definition~\ref{def:H-proofs}). For $N=1$, we are left with the theory of non-iterated (least) fixed points, which will be analyzed in the next section.

In view of Definition~\ref{def:ranks-big-system}, an $\mathcal L^\Omega_\omega$-sentence~$\psi$ is an $\mathcal L^\Omega_N$-sentence precisely if it has rank~$|\psi|\prec\Omega+\omega\cdot N$. So the idea from the previous paragraph can be cast into the following result. Its proof will occupy us for the rest of this section.

\begin{theorem}\label{thm:hat-elim}
    For any $\mathcal L_N^\Omega$-sequent $\Delta$ with $N>0$ we have
    \begin{equation*}
        \mathcal H\vdash^\gamma_{\Omega+\omega\cdot(N+1)}\Delta\qquad\Rightarrow\qquad\mathcal H\vdash^{\varphi_\delta(\delta)}_{\Omega+\omega\cdot N}\Delta\quad\text{for}\quad\delta=\varepsilon_{\Omega+\gamma}.
    \end{equation*}
\end{theorem}

Note that the theorem involves the Veblen function~$\varphi$, which we have introduced in Definition~\ref{def:Gamma}. The $\varepsilon$-function is explained by $\varepsilon_\gamma=\varphi_1(\gamma)$. At least intuitively, this means that $\varepsilon_\gamma$ is the $\gamma$-th fixed point of the function~$\eta\mapsto\omega^\eta=\varphi_0(\eta)$.

The Veblenian bound indicates that we are concerned with predicative ordinal analysis. Indeed, the present section is analogous to an ordinal analysis of the pre\-dica\-tive theory~$\widehat{\mathsf{ID}}_{<\omega}$ (see the explanations and references in Section~\ref{sect:ind-def}). A main technical ingredient is the so-called asymmetric interpretation (see~\cite{cantini-asymmetric,jaeger-strahm-pos-ind}). Even though the approach is classical, it takes care to adapt it to our setting. We thus provide full details.

Let us begin with a cut elimination result that already appears in the ordinal analysis of Peano arithmetic. We include a (condensed) proof for the convenience of the reader and to show that the procedure is compatible with operator control.

\begin{lemma}[$\mathsf{ATR}_0$]\label{lem:basic-cut-elim}
    When $\rho\in\Gamma_{\Omega+Y}$ is not of the form~$\Omega+\omega\cdot n$ with $n\prec\omega$ and for any $\mathcal L^\Omega_\omega$-sequent~$\Delta$, we have
    \begin{equation*}
        \mathcal H\vdash^\alpha_{\rho+1}\Delta\quad\Rightarrow\quad\mathcal H\vdash^{\omega^\alpha}_\rho\Delta.
    \end{equation*}
\end{lemma}
\begin{proof}
    The first ingredient of the argument is the so-called inversion lemma. For a conjunctive formula~$\varphi\simeq\bigwedge_{i\prec\iota(\varphi)}\varphi_i$, it states that we have
    \begin{equation*}
        \mathcal H\vdash^\alpha_\rho\Delta,\varphi\qquad\Rightarrow\qquad\mathcal H[i]\vdash^\alpha_\rho\Delta,\varphi_i\quad\text{for all}\quad i\prec\iota(\varphi).
    \end{equation*}
    Crucially, $\varphi$ cannot have the form $I^0_{\varphi,\Omega}t$ (as this formula is disjunctive) or the form $I^n_\varphi t$ or $\neg I^n_\varphi t$ with $n>0$ (as these formulas are neither conjunctive nor disjunctive). So at least intuitively, the only way to introduce $\varphi$ according to Definition~\ref{def:H-proofs} is clause~($\bigwedge$). Formally, one proves the inversion lemma by induction over~$\alpha$. In~the induction step, first observe that we get $k(\varphi_i)\subseteq k(\varphi)\cup\{i\}\subseteq\mathcal H[i](\emptyset)$. Then distinguish cases according to the clauses from Definition~\ref{def:H-proofs}. In the most interesting case, we have an application of~($\bigwedge$) with premises
    \begin{equation*}
        \mathcal H[i]\vdash^{\alpha(i)}_\rho\Delta,\varphi,\varphi_i.
    \end{equation*}
    We can conclude by the induction hypothesis and weakening (noting that $i\in\mathcal H[i](\emptyset)$ entails $\mathcal H[i][i]=\mathcal H[i]$). Let us point out that~($\bigwedge$) can also be applied to a different conjunctive formula~$\psi\in\Delta$, in which case each $j\prec\iota(\psi)$ has a premise
    \begin{equation*}
        \mathcal H[j]\vdash^{\alpha(j)}_\rho\Delta,\varphi,\psi_j.
    \end{equation*}
    Given~$i\prec\iota(\varphi)$, we inductively have $\mathcal H[i][j]\vdash^{\alpha(j)}_\rho\Delta,\varphi_i,\psi_j$ for all~$j\prec\iota(\psi)$, so that we can reapply~($\bigwedge$) to get $\mathcal H[i]\vdash^\alpha_\rho\Delta,\varphi_i$. The remaining cases are left to the reader.

    The next step is the so-called reduction lemma. For disjunctive~$\varphi\simeq\bigvee_{i\prec\iota(\varphi)}\varphi_i$ of rank~$|\varphi|=\rho\neq\Omega$, it asserts
    \begin{equation*}
        \mathcal H\vdash^\alpha_\rho\Delta,\neg\varphi\quad\text{and}\quad\mathcal H\vdash^\beta_\rho\Delta,\varphi\qquad\Rightarrow\qquad\mathcal H\vdash^{\alpha+\beta}_\rho\Delta.
    \end{equation*}
    Note that we cannot conclude by~($\mathsf{Cut}$), as this would require $|\varphi|\prec\rho$. Thus the reduction lemma shows how a single cut can be avoided. One proves it by induction on~$\beta$. In the induction step, distinguish cases according to the clause of Definition~\ref{def:H-proofs} that was used to obtain~$\mathcal H\vdash^\beta_\rho\Delta,\varphi$. The point is that $\varphi$ cannot be of the form~$I^0_{\theta,\Omega}t$ (as this formula has rank~$\Omega$) or of the form~$I^n_\theta t$ or~$\neg I^n_\theta t$ (as these formulas are not disjunctive), so that there can be no `critical' applications of the rules~($\mathsf F^\Omega$) as well as ($\mathsf F^+$) and ($\mathsf F^-$) with `principal' formula~$\varphi$ (though these rules may apply to `side' formulas in~$\Delta$). In the only interesting case that remains, $\varphi$ has been introduced by~($\bigvee$) with a premise
    \begin{equation*}
        \mathcal H\vdash^{\beta(i)}_\rho\Delta,\varphi,\varphi_i
    \end{equation*}
    for some~$i\prec\min(\iota(\varphi),\beta)$ with $i\in\mathcal H(\emptyset)$ and with~$\beta(i)\prec\beta$. By induction hypothesis (in conjunction with weakening), we obtain
    \begin{equation*}
        \mathcal H\vdash^{\alpha+\beta(i)}_\rho\Delta,\varphi_i.
    \end{equation*}
    Now we cannot reapply~($\bigvee$), because~$\varphi$ is in general not contained in~$\Delta$. Instead, we use inversion on the conjunctive formula $\neg\varphi\simeq\bigwedge_{i\prec\iota(\varphi)}\neg\varphi_i$ in the given proof of $\mathcal H\vdash^\alpha_\rho\Delta,\neg\varphi$. Due to~$\mathcal H[i]=\mathcal H$, this gives
    \begin{equation*}
        \mathcal H\vdash^\alpha_\rho\Delta,\neg\varphi_i.
    \end{equation*}
    Lemma~\ref{lem:rank-subformulas} ensures~$|\varphi_i|\prec|\varphi|=\rho$. Thus we get $\mathcal H\vdash^{\alpha+\beta}_\rho\Delta$ by ($\mathsf{Cut}$).

    Now the implication from the lemma (i.e., the cut elimination result) can be derived by induction on~$\alpha$. Again, we distinguish cases according to the clauses from Definition~\ref{def:H-proofs}. In the interesting case, we are concerned with a cut over a formula~$\varphi$ with rank~$|\varphi|=\rho$, which has premises
    \begin{equation*}
        \mathcal H\vdash^{\alpha(0)}_\rho\Delta,\neg\varphi\quad\text{and}\quad\mathcal H\vdash^{\alpha(1)}_\rho\Delta,\varphi
    \end{equation*}
    for some~$\alpha(i)\prec\alpha$. The assumption on~$\rho$ ensures that~$\varphi$ does not have the form $I^n_\theta t$ or~$\neg I^n_\theta t$ with $n>0$. In view of Definition~\ref{def:disj}, this means that either~$\varphi$ or~$\neg\varphi$ is disjunctive. We may assume that the disjunctive formula is~$\varphi$. By the induction hypothesis and reduction, we get
    \begin{equation*}
        \mathcal H\vdash^{\omega^{\alpha(0)}+\omega^{\alpha(1)}}_\rho\Delta.
    \end{equation*}
    The claim follows because of $\omega^{\alpha(0)}+\omega^{\alpha(1)}\prec\omega^\alpha$ and $\omega^\alpha=\varphi_0(\alpha)\in\mathcal H(\emptyset)$, where the latter reduces to $\alpha\in\mathcal H(\emptyset)$ due to Remark~\ref{rmk:H-closure}.
\end{proof}

It is standard to derive the following result on partial cut elimination, which contributes to the proof of Theorem~\ref{thm:hat-elim}.

\begin{corollary}[$\mathsf{ATR}_0$]\label{cor:cut-elim-eps}
    For any $N\prec\omega$ and any $\mathcal L_\omega^\Omega$-sequent~$\Delta$, we have
    \begin{equation*}
        \mathcal H\vdash^\gamma_{\Omega+\omega\cdot(N+1)}\Delta\quad\Rightarrow\quad\mathcal H\vdash^{\varepsilon_\gamma}_{\Omega+\omega\cdot N+1}\Delta.
    \end{equation*}
\end{corollary}
\begin{proof}
    We argue by induction on~$\gamma$ and distinguish cases according to the clauses from Definition~\ref{def:H-proofs}. In the crucial case, we have an application of~($\mathsf{Cut}$) for some formula~$\varphi$ with~$|\varphi|\prec\Omega+\omega\cdot(N+1)$. If the premises are proved with heights~$\gamma(i)\prec\gamma$, we inductively get
    \begin{equation*}
        \mathcal H\vdash^{\varepsilon_{\gamma(0)}}_{\Omega+\omega\cdot N+1}\Delta,\neg\varphi\quad\text{and}\quad\mathcal H\vdash^{\varepsilon_{\gamma(1)}}_{\Omega+\omega\cdot N+1}\Delta,\varphi.
    \end{equation*}
    Consider~$n\prec\omega$ such that we have $|\varphi|\prec\Omega+\omega\cdot N+1+n$. For $\delta=\max(\gamma(0),\gamma(1))$, an application of~($\mathsf{Cut}$) yields
    \begin{equation*}
        \mathcal H\vdash^{\varepsilon_\delta+1}_{\Omega+\omega\cdot N+1+n}\Delta.
    \end{equation*}
    By~$n$ applications of the previous result, we get
    \begin{equation*}
        \mathcal H\vdash^{\xi(n)}_{\Omega+\omega\cdot N+1}\Delta\quad\text{with}\quad\xi(0)=\varepsilon_\delta+1\text{ and }\xi(k+1)=\omega^{\xi(k)}.
    \end{equation*}
    Given $\xi(k)\prec\varepsilon_\gamma=\varphi_1(\gamma)$, we use Proposition~\ref{prop:Veblen-order} to get $\xi(k+1)=\varphi_0(\xi(k))\prec\varepsilon_\gamma$. This inductively yields $\xi(n)\prec\varepsilon_\gamma$, so that the claim follows by weakening.
\end{proof}

In order to complete the proof of Theorem~\ref{thm:hat-elim}, we need to lower the cut rank from $\Omega+\omega\cdot N+1$ to $\Omega+\omega\cdot N$. In other words, we must eliminate cuts over formulas~$(\neg)I^N_\varphi t$, which may be introduced by the rules~($\mathsf F^+$) and~($\mathsf F^-$). In these rules, the rank of the premise $(\neg)\varphi[t,I^N_\varphi]$ is typically higher than the rank of the conclusion~$(\neg)I^N_\varphi t$. For this reason, the previous considerations no longer apply (recall how we used $|\varphi_i|\prec|\varphi|$ to get reduction in the proof of Lemma~\ref{lem:basic-cut-elim}).

To overcome the obstacle that was explained in the previous paragraph, we intro\-duce an intermediate proof system. For a number $N>0$, let $\mathcal L^+_N$ be the language that results from $\mathcal L^\Omega_{N+1}$ when each predicate~$I^N_\varphi$ is replaced by predicates~$I^N_{\varphi,\alpha}$ for all~$\alpha\in\Gamma_{\Omega+Y}$. One can also view $\mathcal L^+_N$ as the extension of $\mathcal L^\Omega_N$ by the latter predicates.

\begin{definition}\label{def:ranks-intermediate-system}
    The rank~$|\psi|_N$ of an $\mathcal L_N^+$-formula~$\psi$ is recursively given by
    \begin{gather*}
        |\theta|_N=0\quad\text{when $\theta$ is a literal of }\mathcal L_0,\\
        |I^0_{\varphi,\alpha}t|_N=|\neg I^0_{\varphi,\alpha}t|_N=\omega\cdot(1+\alpha),\\
        |I^n_\varphi t|_N=|\neg I^n_\varphi t|_N=\Omega+\omega\cdot n\quad\text{when }0<n<N,\\
        |I^N_{\varphi,\alpha} t|_N=|\neg I^N_{\varphi,\alpha} t|_N=\Omega+\omega\cdot(N+\alpha),\\
        |\psi_0\land\psi_1|_N=|\psi_0\lor\psi_1|_N=\max(|\psi_0|_N,|\psi_1|_N)+1,\\
        |\forall x\psi|_N=|\exists x\psi|_N=|\psi|_N+1.
    \end{gather*}
\end{definition}

Note that the previous definition and Definition~\ref{def:ranks-big-system} assign the same ranks to all $\mathcal L^\Omega_N$-formulas. In view of this fact, we will also write $|\psi|$ at the place of~$|\psi|_N$ when the number $N$ is clear from the context. Similarly, the following is an extension of Definition~\ref{def:disj} from $\mathcal L_N^\Omega$-sentences to~$\mathcal L^+_N$-sentences.

\begin{definition}\label{def:L+-disjunctions}
    Each false literal of first-order arithmetic and each formula~$\neg\overline Qt$ is associated with an empty disjunction. Other than that, the $\mathcal L_N^+$-sentences of~dis\-junctive type are given by
    \begin{align*}
        \psi_0\lor\psi_1&\simeq\textstyle\bigvee_{i\prec2}\psi_i,&\quad I^0_{\varphi,\alpha}t&\simeq\textstyle\bigvee_{\beta\prec\alpha}\varphi[t,I^0_{\varphi,\beta}],\\
        \exists x\,\psi(x)&\simeq\textstyle\bigvee_{i\prec\omega}\psi(i),&\quad I^N_{\varphi,\alpha}t&\simeq\textstyle\bigvee_{\beta\prec\alpha}\varphi[t,I^N_{\varphi,\beta}].
    \end{align*}
\end{definition}

Concerning the clause for $I^N_{\varphi,\alpha}t$ in the above definition, it is crucial to note that $\varphi[t,I^N_{\varphi,\beta}]$ is indeed an $\mathcal L^+_N$-sentence. This is because $\varphi$ is an $\mathcal L_N^X$-formula, which means that predicates $I^n_\theta$ occur only for~$n<N$. Lemma~\ref{lem:rank-subformulas} remains valid:

\begin{lemma}[$\mathsf{ATR}_0$]
    If $\psi$ is an $\mathcal L_N^+$-sentence with $N>0$, we have
    \begin{equation*}
        |\psi_i|_N\prec|\psi|_N\quad\text{for all}\quad i\prec\iota(\varphi).
    \end{equation*}
\end{lemma}
\begin{proof}
    In the only new case, $\psi$ is of the form~$I^N_{\varphi,\alpha}t$. Here $\varphi$ is an $\mathcal L_N^X$-formula, which means that subformulas~$I^n_\theta s$ occur only for~$n<N$. These subformulas have ranks~$\Omega+\omega\cdot n\prec\Omega+\omega\cdot N$ (note especially~$|I^0_{\theta,\Omega}s|=\omega\cdot(1+\Omega)=\Omega$). When we have $\beta\prec\iota(\psi)=\alpha$, we thus get
    \begin{equation*}
        |\psi_\beta|=|\varphi[t,I^N_{\varphi,\beta}]|\prec\Omega+\omega\cdot(N+\beta)+\omega\preceq\Omega+\omega\cdot(N+\alpha)=|I^N_{\varphi,\alpha}t|=|\psi|,
    \end{equation*}
    as desired.
\end{proof}

The following definition extends Definition~\ref{def:formula-coeffs} in the sense that both definitions assign the same coefficients to any $\mathcal L_N^\Omega$-formula.

\begin{definition}\label{def:formula-coeffs-plus}
For an $\mathcal L_N^+$-sentence~$\psi$ with $N>0$, we put
\begin{equation*}
    k(\psi)=\{\alpha:\text{a predicate symbol }I^0_{\varphi,\alpha}\text{ or }I^N_{\varphi,\alpha}\text{ occurs in $\psi$}\}.
\end{equation*}
When $\Delta$ is an $\mathcal L_N^+$-sequent, we set $k(\Delta)=\bigcup_{\psi\in\Delta}k(\psi)$.
\end{definition}

The easiest way to understand the following definition is probably to consider the changes with respect to Definition~\ref{def:H-proofs} (see the remark below). In order to avoid any ambiguity, we nevertheless give the definition in full (as with the previous definition). Note how the notation allows us to distinguish the old proof system from the new one ($\mathcal H\vdash\ldots$ against $\mathcal H,N\vdash\ldots$).

\begin{definition}\label{def:H-proofs-plus}
    Consider~$N>0$. For an operator~$\mathcal H$ and $\alpha,\rho\in\Gamma_{\Omega+Y}$ as well as an $\mathcal L_N^+$-sequent~$\Delta$, we declare that $\mathcal H,N\vdash^\alpha_\rho\Delta$ holds when we have $\{\alpha\}\cup k(\Delta)\subseteq\mathcal H(\emptyset)$ and one of the following recursive clauses applies:
    \begin{enumerate}[label=(\roman*)]
    \item[(\textsf{Ax})] The sequent~$\Delta$ contains formulas $I^n_\varphi s$ and $\neg I^n_\varphi t$ for terms~$s$ and $t$ with the same value (where $0<n<N$).
    \item[$(\bigwedge)$] There is a conjunctive $\varphi\simeq \bigwedge_{i\prec\iota(\varphi)}\varphi_i\in \Delta$ such that $\mathcal H[i],N\vdash^{\alpha(i)}_\rho\Delta,\varphi_i$ holds for all~$i\prec\iota(\varphi)$ with some~$\alpha(i)\prec\alpha$.
    \item[$(\bigvee)$] There is a disjunctive $\varphi\simeq \bigvee_{i\prec\iota(\varphi)}\varphi_i\in\Delta$ such that $\mathcal H,N\vdash^{\beta}_\rho\Delta, \varphi_i$ holds for some~$i\prec\min(\iota(\varphi),\alpha)$ with $i\in\mathcal H(\emptyset)$ and some $\beta\prec\alpha$.
    \item[(\textsf{Cut})] There is an $\mathcal L_N^+$-sentence~$\psi$ with $|\psi|\prec\rho$ such that we have  $\mathcal H,N\vdash^{\alpha(0)}_\rho\Delta,\psi$ and $\mathcal H,N\vdash^{\alpha(1)}_\rho\Delta,\neg\psi$ for some $\alpha(0),\alpha(1)\prec\alpha$.
    \item[($\mathsf F^\Omega$)] We have $\Omega\preceq\alpha$ and there is $I^0_{\varphi,\Omega}t\in\Delta$ such that $\mathcal H,N\vdash^{\beta}_\rho\Delta,\varphi[t,I^0_{\varphi,\Omega}]$ holds for some~$\beta\prec\alpha$.
    \item[($\mathsf F^+$)] There is $I^n_\varphi\, t\in\Delta$ such that we have $\mathcal H,N\vdash^{\beta}_\rho\Delta,\varphi[t,I^n_\varphi]$ for some~$\beta\prec\alpha$ (where $0<n<N$).
    \item[($\mathsf F^-$)] There is $\neg I^n_\varphi\, t\in\Delta$ such that we have $\mathcal H,N\vdash^{\beta}_\rho\Delta,\neg\varphi[t,I^n_\varphi]$ for some~$\beta\prec\alpha$ (where again $0<n<N$).
\end{enumerate}
\end{definition}

As promised, we indicate the changes with respect to the previous system:

\begin{remark}
    The above definition modifies Definition~\ref{def:H-proofs} insofar as 
    \begin{itemize}
        \item it restricts all clauses to the new language, so that ($\mathsf{Ax}$) and ($\mathsf F^+$) and ($\mathsf F^-$) apply only for $n<N$, while ($\mathsf{Cut}$) applies only to $\mathcal L_N^+$-sentences,
        \item the clauses ($\bigwedge$) and ($\bigvee$) apply to all conjunctions and disjunctions that were specified in Definition~\ref{def:L+-disjunctions}, which now includes the formulas $(\neg)I^N_{\varphi,\alpha}t$.
    \end{itemize}
\end{remark}

In one direction, the connection with our previous proof system is made by the following conservativity result.

\begin{lemma}[$\mathsf{ATR}_0$]\label{lem:Vdash-conservative}
    For $N>0$ and any $\mathcal L_N^\Omega$-sequent~$\Delta$, we have
    \begin{equation*}
        \mathcal H,N\vdash^\alpha_{\Omega+\omega\cdot N}\Delta\quad\Rightarrow\quad\mathcal H\vdash^\alpha_{\Omega+\omega\cdot N}\Delta.
    \end{equation*}
\end{lemma}
\begin{proof}
    In view of the cut rank, the proof in the premise contains no predicates $I^N_{\varphi,\gamma}$. As a consequence, all instances of ($\bigwedge$) and ($\bigvee$) in this proof are also allowed by the proof system in the conclusion. Formally, one argues by induction on~$\alpha$.
\end{proof}

Conversely, we want to show that the conclusion of Corollary~\ref{cor:cut-elim-eps} (specifically a proof $\mathcal H\vdash^\delta_{\Omega+\omega\cdot N+1}\Delta$ for an $\mathcal L_N^\omega$-sequent~$\Delta$) can be transformed into a proof in our new system. For this purpose, we will employ a so-called asymmetric interpretation (see~\cite{cantini-asymmetric,jaeger-strahm-pos-ind}), which is based on the following notion of relativization.

\begin{definition}\label{def:asymmetric-interpretation}
    For $N>0$ and~$\alpha,\beta\in\Gamma_{\Omega+Y}$, we transform each $\mathcal L_{N+1}^\Omega$-formula~$\varphi$ into the $\mathcal L_N^+$-formula~$\varphi^{\alpha,\beta}$ that is recursively given by
    \begin{gather*}
        (P\mathbf t)^{\alpha,\beta}=P\mathbf t\quad\text{and}\quad(\neg P\mathbf t)^{\alpha,\beta}=\neg P\mathbf t\quad\text{for any predicate~$P$ of~$\mathcal L_N^\Omega$},\\
        \begin{alignedat}{3}
        (I^N_\varphi\, t)^{\alpha,\beta}&=I^N_{\varphi,\beta}\, t\quad&&\text{and}\quad
        (\neg I^N_\varphi\, t)^{\alpha,\beta}&&=\neg I^N_{\varphi,\alpha}\, t,\\
        (\varphi\land\psi)^{\alpha,\beta}&=\varphi^{\alpha,\beta}\land\psi^{\alpha,\beta}\quad&&\text{and}\quad
        (\varphi\lor\psi)^{\alpha,\beta}&&=\varphi^{\alpha,\beta}\lor\psi^{\alpha,\beta},\\
        (\forall x\,\varphi)^{\alpha,\beta}&=\forall x\,\varphi^{\alpha,\beta}\quad&&\text{and}\quad(\exists x\,\varphi)^{\alpha,\beta}&&=\exists x\,\varphi^{\alpha,\beta}.
        \end{alignedat}
    \end{gather*}
When $\Delta$ is an $\mathcal L_{N+1}^+$-sequent, we put $\Delta^{\alpha,\beta}=\{\varphi^{\alpha,\beta}:\varphi\in\Delta\}.$
\end{definition}

More intuitively, we get $\varphi^{\alpha,\beta}$ from~$\varphi$ when we replace all negative and positive occurrences of predicates~$I^N_\varphi$ by the $\alpha$-th and $\beta$-th approximations, respectively. In view of Definition~\ref{def:L+-disjunctions}, these can be understood as occurrences of ordinal quantifiers, which are relativized to $\alpha$ and $\beta$ in the universal and existential case. This motivates the following monotonicity result.

\begin{lemma}[$\mathsf{ATR}_0$]\label{lem:asym-monotone}
    Consider an~$\mathcal L_{N+1}^\Omega$-sequent~$\Delta$ and an~$\mathcal L_N^+$-sequent~$\Xi$ for some $N>0$. For any $\gamma,\delta\in\mathcal H(\emptyset)$ with $\alpha\succeq\gamma$ and $\beta\preceq\delta$, we have
    \begin{equation*}
        \mathcal H,N\vdash^\eta_\rho\Delta^{\alpha,\beta},\Xi\qquad\Rightarrow\qquad \mathcal H,N\vdash^\eta_\rho\Delta^{\gamma,\delta},\Xi.
    \end{equation*}
\end{lemma}
\begin{proof}
    We argue by induction on~$\eta$. First note that we get
    \begin{equation*}
        k\left(\Delta^{\gamma,\delta}\right)\subseteq k\left(\Delta^{\alpha,\beta}\right)\cup\{\gamma,\delta\}\subseteq\mathcal H(\emptyset).
    \end{equation*}
    We now distinguish cases according to the clauses from Definition~\ref{def:H-proofs-plus}. Consider the case where clause ($\bigvee$) was used to introduce $\varphi\simeq\bigvee_{i\prec\iota(\varphi)}\varphi_i\in\Delta^{\alpha,\beta}$ with a premise
    \begin{equation*}
        \mathcal H,N\vdash^{\xi}_\rho\Delta^{\alpha,\beta},\varphi_i,\Xi,
    \end{equation*}
    where we have $i\prec\min(\iota(\varphi),\eta)$ with $i\in\mathcal H(\emptyset)$ and $\xi\prec\eta$. The sentence $\varphi$ is given as $\psi^{\alpha,\beta}$ for some~$\psi\in\Delta$. First assume that $\psi$ has the form~$I^N_\theta t$, so that we have $\varphi\equiv I^N_{\theta,\beta}t$ and hence $i\prec \iota(\varphi)=\beta$ as well as~$\varphi_i\equiv\theta[t,I^N_{\theta,i}]$. We inductively get
    \begin{equation*}
        \mathcal H,N\vdash^{\xi}_\rho\Delta^{\gamma,\delta},\theta[t,I^N_{\theta,i}],\Xi.
    \end{equation*}
    Observe that we have
    \begin{equation*}
        \psi^{\gamma,\delta}\equiv I^N_{\theta,\delta}t\simeq\textstyle\bigvee_{j\prec\delta}\theta[t,I^n_{\theta,j}]\in\Delta^{\gamma,\delta}.
    \end{equation*}
    In view of $i\prec\beta\preceq\delta=\iota(\psi^{\gamma,\delta})$, we can reapply~($\bigvee$) to get $\mathcal H,N\vdash^\eta_\rho\Delta^{\gamma,\delta},\Xi$.
    
    Still concerning~($\bigvee$) with $\varphi\equiv\psi^{\alpha,\beta}$, we now assume that $\psi$ is not of the form~$I^N_\theta t$. Given that $\psi^{\alpha,\beta}$ is disjunctive in the sense of Definition~\ref{def:L+-disjunctions}, this means that $\psi$ has the form $\psi_0\lor\psi_1$ or $\exists x\,\psi'$ or $I^0_{\theta,\tau} t$. In each of these cases, $\psi$ itself is disjunctive in the sense of Definition~\ref{def:disj} and we have
    \begin{equation*}
        \left(\psi^{\alpha,\beta}\right)_j\equiv(\psi_j)^{\alpha,\beta}\quad\text{for}\quad j\prec\iota\left(\psi^{\alpha,\beta}\right)=\iota(\psi).
    \end{equation*}
    Specifically for $\psi\equiv I^0_{\theta,\tau}t$ with $\psi_j\equiv\theta[t,I^0_{\theta,j}]$, we note that we have $(\psi_j)^{\alpha,\beta}\equiv\psi_j$, because $\theta$ is an $\mathcal L_0^X$-formula. In all relevant cases, we can now write $\psi_j^{\alpha,\beta}$ without parentheses (also with $\gamma,\delta$ at the place of~$\alpha,\beta$). Since $\varphi_i$ is $\psi_i^{\alpha,\beta}$, we inductively~get
    \begin{equation*}
        \mathcal H,N\vdash^{\xi}_\rho\Delta^{\gamma,\delta},\psi_i^{\gamma,\delta},\Xi.
    \end{equation*}
    In view of $\psi^{\gamma,\delta}\in\Delta^{\gamma,\delta}$, we can reapply~($\bigvee$) to conclude.

    The case where ($\bigwedge$) introduces a conjunctive formula in~$\Delta^{\alpha,\beta}$ is treated similarly. In any other case, it is straightforward to reduce to the induction hypothesis (as the formulas that are relevant for the inference are unaffected by the relativization).
\end{proof}

We can now examine in detail how the proof system from Definition~\ref{def:H-proofs} is rendered into the one from Definition~\ref{def:H-proofs-plus} using the asymmetric interpretation.

\begin{proposition}[$\mathsf{ATR}_0$]\label{prop:asym-int}
    For $\alpha\in\mathcal H(\emptyset)$ and any $\mathcal L_{N+1}^\Omega$-sequent~$\Delta$ with $N>0$, we have
    \begin{equation*}
        \mathcal H\vdash^\beta_{\Omega+\omega\cdot N+1}\Delta\qquad\Rightarrow\qquad \mathcal H,N\vdash^{\omega\cdot\gamma}_\gamma\Delta^{\alpha,\alpha+\omega^\beta}.
    \end{equation*}
    with $\gamma=\gamma(\alpha,\beta)=\Omega+\omega\cdot(N+\alpha+\omega^\beta)$.
\end{proposition}
\begin{proof}
    We argue by induction on~$\beta$. First note that the premise of the desired implication ensures $\{\beta\}\cup k(\Delta)\subseteq\mathcal H(\emptyset)$. Together with the assumption~$\alpha\in\mathcal H(\emptyset)$, we get $k(\Delta^{\alpha,\alpha+\omega^\beta})\subseteq k(\Delta)\cup\{\alpha,\alpha+\omega^\beta\}\subseteq\mathcal H(\emptyset)$ and also $\gamma=\gamma(\alpha,\beta)\in\mathcal H(\emptyset)$.

    Let us now distinguish cases according to the clauses from Definition~\ref{def:H-proofs}. First assume that we have an instance of ($\mathsf{Ax}$), so that $\Delta$ contains formulas~$I^n_\varphi s$ and~$\neg I^n_\varphi t$ for terms~$s$ and $t$ with the same value. As $\Delta$ is an $\mathcal L_{N+1}^\Omega$-sequent, we have $n\leq N$. If the latter inequality is strict, the formulas $I^n_\varphi s$ and~$\neg I^n_\varphi t$ are still in~$\Delta^{\alpha,\alpha+\omega^\beta}$ and we can conclude by clause~($\mathsf{Ax}$) of Definition~\ref{def:H-proofs-plus}. Now assume that we have~$n=N$. The sequent~$\Delta^{\alpha,\alpha}$ then contains both~$I^N_{\varphi,\alpha}s$ and $\neg I^N_{\varphi,\alpha}t$. Analogous to Lemma~\ref{lem:psi-neg-psi}, we get
    \begin{equation*}
        \mathcal H,N\vdash^{\omega\cdot\delta}_0 I^N_{\varphi,\alpha}s,\neg I^N_{\varphi,\alpha}s\quad\text{with}\quad\delta=|I^N_{\varphi,\alpha}s|=\Omega+\omega\cdot(N+\alpha)\prec\gamma(\alpha,\beta)=:\gamma.
    \end{equation*}
    By the analogs of Lemma~\ref{lem:weakening} (weakening) and Remark~\ref{rmk:term-value-variant} (numerical variants), we arrive at~$\mathcal H,N\vdash^{\omega\cdot\gamma}_{\gamma}\Delta^{\alpha,\alpha}$. The claim follows by the monotonicity property from the previous lemma.

    Next, we consider the case where a sentence~$\varphi\simeq\bigwedge_{i\prec\iota(\varphi)}\varphi_i\in\Delta$ has been introduced by clause~($\bigwedge$). Here we have premises
    \begin{equation*}
        \mathcal H[i]\vdash^{\beta(i)}_{\Omega+\omega\cdot N+1}\Delta,\varphi_i
    \end{equation*}
    with $\beta(i)\prec\beta$ for~$i\prec\iota(\varphi)$. Writing $\gamma_i=\gamma(\alpha,\beta(i))$, we can use the induction hypothesis and monotonicity to get
    \begin{equation*}
        \mathcal H[i],N\vdash^{\omega\cdot\gamma_i}_\gamma\Delta^{\alpha,\alpha+\omega^\beta},(\varphi_i)^{\alpha,\alpha+\omega^\beta}.
    \end{equation*}
    Given that $\varphi$ is conjunctive in the sense of Definition~\ref{def:disj}, we see that $(\varphi_i)^{\alpha,\delta}$ coincides with $(\varphi^\delta)_i$ for all~$i\prec\iota(\varphi)=\iota(\varphi^{\alpha,\delta})$, as in the previous proof. We can thus reapply~($\bigwedge$) to obtain the claim. An analogous argument applies to clause~($\bigvee$).

    We now consider an application of~($\mathsf{Cut}$) with premises
    \begin{equation*}
        \mathcal H\vdash^{\beta(0)}_{\Omega+\omega\cdot N+1}\Delta,\psi\quad\text{and}\quad\mathcal H\vdash^{\beta(1)}_{\Omega+\omega\cdot N+1}\Delta,\neg\psi
    \end{equation*}
    for $\beta(i)\prec\beta$, where $\psi$ is an $\mathcal L_\omega^\Omega$-sentence of rank $|\psi|\prec\Omega+\omega\cdot N+1$. The latter ensures that $\psi$ is an $\mathcal L_{N+1}^\Omega$-sentence, so that the induction hypothesis applies. More precisely, $\psi$ is either an $\mathcal L_N^\Omega$-sentence or of the form~$I^N_\varphi t$ (or $\neg I^N_\varphi t$ but then we swap with~$\neg\psi$). Write $\gamma=\gamma(\alpha,\beta)$ as well as
    \begin{equation*}
        \gamma_0=\gamma(\alpha,\beta(0))\quad\text{and}\quad\gamma_1=\gamma(\overline\alpha,\beta(1))\quad\text{with}\quad\overline\alpha=\alpha+\omega^{\beta(0)},
    \end{equation*}
    where $\overline\alpha+\omega^{\beta(1)}\prec\alpha+\omega^\beta$ entails~$\gamma_i\prec\gamma$. Using the induction hypothesis (for $\alpha$ and $\overline\alpha$) in conjunction with monotonicity, we get
    \begin{equation*}
        \mathcal H,N\vdash^{\omega\cdot\gamma_0}_\gamma\Delta^{\alpha,\alpha+\omega^\beta},\psi^{\alpha,\overline\alpha}\quad\text{and}\quad\mathcal H,N\vdash^{\omega\cdot\gamma_1}_\gamma\Delta^{\alpha,\alpha+\omega^\beta},(\neg\psi)^{\overline\alpha,\alpha+\omega^\beta}.
    \end{equation*}
    Now $\psi^{\alpha,\overline\alpha}$ depends only on the second superscript. Indeed, we have $\psi^{\alpha,\overline\alpha}\equiv\psi$ when $\psi$ is an $\mathcal L_N^\Omega$-sentence and furthermore~$\psi^{\alpha,\overline\alpha}\equiv I^N_{\varphi,\overline\alpha}t$ for $\psi\equiv I^N_\varphi t$, which also shows
    \begin{equation*}
        |\psi^{\alpha,\overline\alpha}|\preceq\Omega+\omega\cdot(N+\overline\alpha)\prec\gamma.
    \end{equation*}
    Similarly, $(\neg\psi)^{\overline\alpha,\alpha+\omega^\beta}$ depends only on the first superscript, so that we have
    \begin{equation*}
        (\neg\psi)^{\overline\alpha,\alpha+\omega^\beta}\equiv\neg\left(\psi^{\alpha,\overline\alpha}\right).
    \end{equation*}
    We thus get $\mathcal H,N\vdash^{\omega\cdot\gamma}_\gamma\Delta^{\alpha,\alpha+\omega^\beta}$ by~($\mathsf{Cut}$). We emphasize that the whole point of the cut rank~$\Omega+\omega\cdot n+1$ in the premise was to ensure that $\psi$ cannot contain both positive and negative occurrences of predicates~$I^N_\varphi t$.

    In the case of clause~($F^\Omega$), one readily reduces to the induction hypothesis, as the conclusion $I^0_{\varphi,\Omega}t$ and the premise~$\varphi[t,I^0_{\varphi,\Omega}]$ are unaffected by the relativization. The same applies to applications of~($\mathsf F^+$) and~($\mathsf F^+$) to formulas~$(\neg)I^n_\varphi t$ with $n<N$. Now consider the case where~($\mathsf F^+$) was used to introduce~$I^N_\varphi t\in\Delta$ with premise
    \begin{equation*}
    \mathcal H\vdash^{\overline\beta}_{\Omega+\omega\cdot N+1}\Delta,\varphi[t,I^N_\varphi]
    \end{equation*}
    for~$\overline\beta\prec\beta$. Given that $\varphi$ is a positive~$\mathcal L_N^X$-formula, the induction hypothesis and monotonicity show that $\overline\gamma=\gamma(\alpha,\overline\beta)$ validates
    \begin{equation*}
        \mathcal H,N\vdash^{\omega\cdot\overline\gamma}_\gamma\Delta^{\alpha,\alpha+\omega^\beta},\varphi\left[t,I^N_{\varphi,\alpha+\omega^{\overline\beta}}\right].
    \end{equation*}
    In view of $\alpha+\omega^{\overline\beta}\prec\alpha+\omega^\beta\preceq\omega\cdot\gamma$ and
    \begin{equation*}
        I^N_{\varphi,\alpha+\omega^\beta}t\simeq\textstyle\bigvee_{i\prec\alpha+\omega^\beta}\varphi[t,I^N_{\varphi,i}]\in\Delta^{\alpha,\alpha+\omega^\beta},
    \end{equation*}
    we can apply~($\bigvee$) to conclude.

    Finally, we consider an application of~($\mathsf F^-$) to~$\neg I_{\varphi}^N t\in\Delta$ that has premise
    \begin{equation*}
    \mathcal H\vdash^{\overline\beta}_{\Omega+\omega\cdot N+1}\Delta,\neg\varphi[t,I^N_\varphi]
    \end{equation*}
    with~$\overline\beta\prec\beta$. For $\overline\gamma$ as above, we inductively get
    \begin{equation*}
        \mathcal H,N\vdash^{\omega\cdot\overline\gamma}_\gamma\Delta^{\alpha,\alpha+\omega^\beta},\neg\varphi[t,I^N_{\varphi,\alpha}].
    \end{equation*}
    For any~$i\prec\alpha$, we can use monotonicity (after weakening~$\mathcal H$ to~$\mathcal H[i]$) in order to get
    \begin{equation*}
        \mathcal H[i],N\vdash^{\omega\cdot\overline\gamma}_\gamma\Delta^{\alpha,\alpha+\omega^\beta},\neg\varphi[t,I^N_{\varphi,i}].
    \end{equation*}
    In view of
    \begin{equation*}
        \neg I^N_{\varphi,\alpha}t\simeq\textstyle\bigwedge_{i\prec\alpha}\neg\varphi[t,I^N_{\varphi,i}]\in\Delta^{\alpha,\alpha+\omega^\beta},
    \end{equation*}
    we can conclude by an application of~($\bigwedge$).
\end{proof}

In particular, we obtain the following conservativity result.

\begin{corollary}[$\mathsf{ATR}_0$]\label{cor:asymmetric}
    For any $\mathcal L_N^\Omega$-sequent~$\Delta$ with~$N>0$, we have
    \begin{equation*}
        \mathcal H\vdash^\beta_{\Omega+\omega\cdot N+1}\Delta\qquad\Rightarrow\qquad\mathcal H,N\vdash^{\omega\cdot\gamma}_\gamma\Delta\quad\text{for}\quad\gamma=\Omega+\omega\cdot N+\omega^{1+\beta}.
    \end{equation*}
\end{corollary}
\begin{proof}
    Apply the previous result with $\alpha=0$ and note that $\Delta$ is unaffected by the asymmetric interpretation, given that it contains predicates~$I^n_\varphi$ for $n<N$ only.
\end{proof}

In order to match the conclusion of the corollary with the premise of Lemma~\ref{lem:Vdash-conservative}, we continue with the following result on predicative cut elimination.

\begin{proposition}[$\mathsf{ATR}_0$]\label{prop:pred-cut-elim}
    For $N>0$ and any $\alpha,\beta\in\mathcal H(\emptyset)$, we have
    \begin{equation*}
        \mathcal H,N\vdash^\gamma_{\Omega+\omega\cdot N+\alpha+\omega^\beta}\Delta\quad\Rightarrow\quad\mathcal H,N\vdash^{\varphi_\beta(\gamma)}_{\Omega+\omega\cdot N+\alpha}\Delta.
    \end{equation*}
\end{proposition}
\begin{proof}
    Analogous to Lemma~\ref{lem:basic-cut-elim}, we have
    \begin{equation*}
    \mathcal H,N\vdash^\gamma_{\rho+1}\Delta\quad\Rightarrow\quad\mathcal H,N\vdash^{\omega^\gamma}_{\rho}\Delta
    \end{equation*}
    for any~$\rho\succeq\Omega+\omega\cdot N$. To explain the latter condition, we note that $\Omega+\omega\cdot N$ bounds the rank of formulas that are the subject of clauses~($F^\Omega$) and ($\mathsf F^+$) and~($\mathsf F^-$) from Definition~\ref{def:H-proofs-plus}. Formally, one derives the claim by induction on~$\gamma$, based on inversion and cut reduction results as in the proof of Lemma~\ref{lem:basic-cut-elim}.

    In order to derive the proposition, one now generalizes the proof of Corollary~\ref{cor:cut-elim-eps}. Specifically, one argues by main induction on~$\beta$ and side induction on~$\gamma$ (or by a single induction on~$(\beta,\gamma)$ with respect to the lexicographic order). First note that we get $\varphi_\beta(\gamma)\in\mathcal H(\emptyset)$ due to Remark~\ref{rmk:H-closure}. We now distinguish cases according to the clauses from Definition~\ref{def:H-proofs-plus}. In the crucial case, we have an application of~($\mathsf{Cut}$) to an~$\mathcal L_N^+$-sentence~$\psi$ of rank $|\psi|\prec\Omega+\omega\cdot N+\alpha+\omega^\beta$. Given that the premises of the cut are derived with heights~$\gamma_i\prec\gamma$, the (side) induction hypothesis yields
    \begin{equation*}
        \mathcal H,N\vdash^{\varphi_\beta(\gamma_0)}_{\Omega+\omega\cdot N+\alpha}\Delta,\psi\quad\text{and}\quad\mathcal H,N\vdash^{\varphi_\beta(\gamma_1)}_{\Omega+\omega\cdot N+\alpha}\Delta,\neg\psi.
    \end{equation*}
    If we have $|\psi|\prec\Omega+\omega\cdot N+\alpha$, an application of ($\mathsf{Cut}$) yields the desired result. Otherwise, it at least gives
    \begin{equation*}
        \mathcal H,N\vdash^{\varphi_\beta(\overline\gamma)+1}_{|\psi|+1}\Delta\quad\text{with}\quad\overline\gamma=\max(\gamma_0,\gamma_1)\prec\gamma.
    \end{equation*}
    By the first paragraph of this proof (which treats the case~$\beta=0$ with $\varphi_\beta(\delta)=\omega^\delta$), we obtain
    \begin{equation*}
        \mathcal H,N\vdash^{\delta_0}_{|\psi|}\Delta\quad\text{with}\quad\delta_0=\varphi_0(\varphi_\beta(\overline\gamma)+1).
    \end{equation*}
    Let us now write
    \begin{equation*}
        |\psi|=\Omega+\omega\cdot N+\alpha+\omega^{\beta_n}+\ldots+\omega^{\beta_1}\quad\text{with}\quad\beta\succ\beta_n\succeq\ldots\succeq\beta_1.
    \end{equation*}
    Due to~$|\psi|\in\mathcal H(k(\psi))=\mathcal H(\emptyset)$, we can use Remark~\ref{rmk:H-closure} to get $\beta_1,\ldots,\beta_n\in\mathcal H(\emptyset)$. By iterative applications of the (main) induction hypothesis (with shorter and shorter sums $\alpha+\omega^{\beta_n}+\ldots+\omega^{\beta_i}$ at the place of~$\alpha$), we thus arrive at
    \begin{equation*}
        \mathcal H,N\vdash^{\delta_n}_{\Omega+\omega\cdot N+\alpha}\Delta\quad\text{with}\quad\delta_n=\varphi_{\beta_n}(\ldots\varphi_{\beta_1}(\delta_0)\ldots).
    \end{equation*}
    More precisely, the proof height is given by the recursive clause $\delta_k=\varphi_{\beta_k}(\delta_{k-1})$ for~$0<k\leq n$. To conclude, we prove~$\delta_k\prec\varphi_\beta(\gamma)$ by induction. Here we may assume~$\beta\succ 0$. The latter ensures that $\overline\gamma\prec\gamma$ entails not only~$\varphi_\beta(\overline\gamma)+1\prec\varphi_\beta(\gamma)$ but then by Proposition~\ref{prop:Veblen-order} also $\delta_0\prec\varphi_\beta(\gamma)$, as needed for the base case. By the same proposition, $\beta_k\prec\beta$ and $\delta_{k-1}\prec\varphi_\beta(\gamma)$ entail~$\delta_k\prec\varphi_\beta(\gamma)$, which accounts for the induction step.    
\end{proof}

We can now derive the theorem that was stated at the beginning of this section. Let us recall that it asserts
\begin{equation*}
\mathcal H\vdash^\gamma_{\Omega+\omega\cdot(N+1)}\Delta\quad\Rightarrow\quad\mathcal H\vdash^{\varphi_\delta(\delta)}_{\Omega+\omega\cdot N}\Delta
\end{equation*}
for $\delta=\varepsilon_{\Omega+\gamma}$, where $\Delta$ can be any~$\mathcal L_N^\Omega$-sequent with~$N>0$.

\begin{proof}[Proof of Theorem~\ref{thm:hat-elim}]
If we have $\gamma=0$, the proof in the premise cannot involve the rule~($\mathsf{Cut}$), so that the result is immediate. We now assume~$\gamma\succ0$. Crucially, this entails $\Omega+\omega\cdot N+\omega^{1+\delta}=\delta$. We get
\begin{alignat*}{3}
    \mathcal H\vdash^\gamma_{\Omega+\omega\cdot(N+1)}\Delta\quad&\Rightarrow\quad \mathcal H\vdash^\delta_{\Omega+\omega\cdot N+1}\Delta\quad&&\Rightarrow\quad\mathcal H,N\vdash^\delta_{\Omega+\omega\cdot N+\omega^\delta}\Delta\\
    {}&\Rightarrow\quad\mathcal H,N\vdash^{\varphi_\delta(\delta)}_{\Omega+\omega\cdot N}\Delta\quad&&\Rightarrow\quad\mathcal H\vdash^{\varphi_\delta(\delta)}_{\Omega+\omega\cdot N}\Delta,
\end{alignat*}
where the implications hold by Corollaries~\ref{cor:cut-elim-eps} and~\ref{cor:asymmetric}, Proposition~\ref{prop:pred-cut-elim} and Lemma~\ref{lem:Vdash-conservative}, respectively.
\end{proof}

Finally, we derive the improved embedding result that was promised at the end of the previous section. The premise of the following theorem refers to the proof system from Definition~\ref{def:inf-proofs}, where $x$ comes from the order~$X$ that was fixed in Assumption~\ref{as:Gamma-Omega-X}. In view of the latter, we get an element $\Gamma_{\Omega+1+x}$ of the order~$\Gamma_{\Omega+Y}$ (recall $Y=1+X$), which may thus appear in the conclusion (see Definition~\ref{def:H-proofs}).

\begin{theorem}[$\mathsf{ATR}_0$]
    For any $\mathcal L_n$-sequent~$\Delta$ with $n>0$ and any operator~$\mathcal H$, we have
    \begin{equation*}
        \vdash^x\Delta\quad\Rightarrow\quad\mathcal H\vdash^{\Gamma_{\Omega+1+x}}_{\Omega+\omega\cdot n}\Delta.
    \end{equation*}
\end{theorem}
\begin{proof}
    Let us first note that we have $\Gamma_{\Omega+1+x}\in\mathcal H(\emptyset)$ by Remark~\ref{rmk:H-closure}. We also have $k(\Delta)\subseteq\{\Omega\}$, where $\Omega\in k(\Delta)$ is possible due to the identification of~$I^0_\varphi$ with $I^0_{\varphi,\Omega}$. In view of $\Omega=\Gamma_{\Omega+0}$ (see the paragraph after Assumption~\ref{as:Gamma-Omega-X}), we get $k(\Delta)\subseteq\mathcal H(\emptyset)$.
    
    We now argue by induction on~$x$ and distinguish cases according to the clauses from Definition~\ref{def:inf-proofs}. Concerning clause~(i), we note that formulas~$\overline Qt$ and true literals of first order arithmetic are empty conjunctions (see Definition~\ref{def:disj}), which can be derived by clause~($\bigwedge$) of Definition~\ref{def:H-proofs}. Instances of the axiom~($\mathsf F$) and ($\mathsf L$) are covered by Propositions~\ref{prop:proof-F} and~\ref{prop:proof-L} (noting $\Omega=\Gamma_{\Omega+0}\prec\Gamma_{\Omega+1+x}$).

    In clause~(ii) of Definition~\ref{def:inf-proofs}, the sequent~$\Delta$ contains sentences~$I^n_\varphi s$ and~$\neg I^n_\varphi t$ for terms~$s$ and~$t$ with the same value. We can conclude by Lemma~\ref{lem:psi-neg-psi} and Remark~\ref{rmk:term-value-variant} (or for~$n>0$ directly by clause~($\mathsf{Ax}$) of Definition~\ref{def:H-proofs}). Clauses~(iii) and~(v) of Definition~\ref{def:inf-proofs} (conjunction and universal quantifier) translate into~($\bigwedge$) of Definition~\ref{def:H-proofs}. Furthermore, clauses~(iv) and~(vi) (disjunction and existential quantifier) translate into clause~($\bigvee$).

    Finally, clause~(vii) of Definition~\ref{def:inf-proofs} concerns a cut with premises
    \begin{equation*}
        \vdash^{x(0)}\Delta,\varphi\quad\text{and}\quad\vdash^{x(1)}\Delta,\neg\varphi
    \end{equation*}
    for $x(i)<_Xx$. Here $\varphi$ can be any $\mathcal L_\omega$-sentence. We choose~$N\prec\omega$ so large that $\varphi$ is an $\mathcal L_{n+N}$-sentence. By the induction hypothesis and a cut, we obtain
    \begin{equation*}
        \mathcal H\vdash^{\delta(0)}_{\Omega+\omega\cdot(n+N)}\Delta\quad\text{with}\quad\delta(0)=\Gamma_{\Omega+1+\max\{x(0),x(1)\}}+1.
    \end{equation*}
    By $N$ applications of Theorem~\ref{thm:hat-elim}, we get $\mathcal H\vdash^{\delta(N)}_{\Omega+\omega\cdot n}\Delta$, where the proof height is recursively determined by
    \begin{equation*}
        \delta(i+1)=\varphi_\eta(\eta)\quad\text{with}\quad\eta=\varepsilon_{\Omega+\delta(i)}=\varphi_1(\Omega+\delta(i)).
    \end{equation*}
    In view of Definition~\ref{def:Gamma}, we inductively get $\delta(i)\prec\Gamma_{\Omega+1+x}$, so that we can conclude by weakening. As an aside, let us note that we have $\Omega+\delta(i)=\delta(i)$ and for $i>0$ even $\varepsilon_{\Omega+\delta(i)}=\delta(i)$, which simplifies the definition of~$\delta(i+1)$.
\end{proof}

In the conclusion of the theorem, the cut rank ensures that only~$\mathcal L_n^\Omega$-formulas occur in the proof. When we have $n=1$, this means that the clauses~($\mathsf F^+$) and~($\mathsf F^-$) cannot apply. The main clause that remains is ($\mathsf F^\Omega$). In the following section, we use methods of impredicative ordinal analysis to deal with this clause. To conclude this section, we optimize the cut rank as follows.

\begin{corollary}[$\mathsf{ATR}_0$; `Embedding']\label{cor:embedding}
For any $\mathcal L_1$-sequent~$\Delta$ and any operator~$\mathcal H$, we have
    \begin{equation*}
        \vdash^x\Delta\quad\Rightarrow\quad\mathcal H\vdash^{\Gamma_{\Omega+1+x}}_{\Omega+1}\Delta.
    \end{equation*}
\end{corollary}
\begin{proof}
    This holds by the previous result in conjunction with Corollary~\ref{cor:cut-elim-eps}, given that we have $\varepsilon_{\Gamma_z}=\varphi_1(\Gamma_z)=\Gamma_z$ by Definition~\ref{def:Gamma}.
\end{proof}

\section{Ordinal analysis: the impredicative part}\label{sect:impred-ord-ana}

In the previous section, we have eliminated detours through~$\mathcal L_{n+1}^\Omega$-sentences from proofs of $\mathcal L_n^\Omega$-sequents for~$n>0$. The present section will achieve the same for~$n=0$. At the end of the section, we will derive the missing direction of our main result, which was stated as Theorem~\ref{thm:main} in the introduction.

The case of $n=0$ is fundamentally different. This is because we had to interpret the predicates~$I^0_\varphi$ by approximations~$I^0_{\varphi,\Omega}$ of a fixed stage~$\Omega$ to derive axiom~($\mathsf L$) in the infinitary proof system (see Proposition~\ref{prop:proof-L}). In contrast, the predicates~$I^n_\varphi$ at levels $n>0$ were interpreted by approximations~$I^n_{\varphi,\alpha}$ in a dynamical way (see the proof of Proposition~\ref{prop:asym-int}). In view of this difference, we cannot eliminate the rule~($\mathsf F^\Omega$) in the way we eliminated~($\mathsf F^+$) and~($\mathsf F^-)$. The rule~($\mathsf F^\Omega$) is a fundamental obstacle to cut elimination in that its premise~$\varphi[t,I^0_{\varphi,\Omega}]$ generally has higher rank than its conclusion~$I^0_{\varphi,\Omega}t$. Rather than the predicative method of cut elimination, the present section will thus employ the impredicative method of collapsing (see~\cite{pohlers98} for a textbook treatment or~\cite{freund-course2} for a presentation that uses our notation).

At the place of the asymmetric interpretation from the previous section, we here use the following notion of relativization. Intuitively, the formula $\psi^\beta$ results from~$\psi$ when all (positive and negative) occurences of ~$I^0_{\varphi,\Omega}$ are replaced by $I^0_{\varphi,\beta}$.

\begin{definition}\label{def:sym-rel}
    Consider $\beta\in\Gamma_{\Omega+Y}$ with $\beta\prec\Omega$. We transform each $\mathcal L_1^\Omega$-formula~$\psi$ into another $\mathcal L_1^\Omega$-formula~$\psi^\beta$ that is recursively given by
    \begin{gather*}
        (P\mathbf t)^\beta=P\mathbf t\quad\text{and}\quad(\neg P\mathbf t)^\beta=\neg P\mathbf t\quad\parbox[t]{.37\textwidth}{when $P$ is a predicate of~$\mathcal L_0$\\ or a predicate $I^0_{\varphi,\alpha}$ with $\alpha\prec\Omega$,}\\
        \begin{alignedat}{3}
        \left(I^0_{\varphi,\Omega}\, t\right)^\beta&=I^0_{\varphi,\beta}\, t\quad&&\text{and}\quad
        \left(\neg I^0_{\varphi,\Omega}\, t\right)^\beta&&=\neg I^0_{\varphi,\beta}\, t,\\
        (\varphi\land\psi)^{\beta}&=\varphi^{\beta}\land\psi^{\beta}\quad&&\text{and}\quad
        (\varphi\lor\psi)^{\beta}&&=\varphi^{\beta}\lor\psi^{\beta},\\
        (\forall x\,\varphi)^{\beta}&=\forall x\,\varphi^{\beta}\quad&&\text{and}\quad(\exists x\,\varphi)^{\beta}&&=\exists x\,\varphi^{\beta}.
        \end{alignedat}
    \end{gather*}
\end{definition}

The following classical result shows, in particular, that the height of a proof bounds the witnesses to existential statements. This is not quite surprising, since the condition $i\prec\alpha$ in clause~($\bigvee$) of Definition~\ref{def:H-proofs} is designed to ensure this.

\begin{proposition}[$\mathsf{ATR}_0$; `Bounding']\label{prop:boundedness}
We consider an $\mathcal L_1^\Omega$-sequent~$\Delta,\psi$ as well as some~$\rho\preceq\Omega+\omega$. For any operator $\mathcal{H}$ and $\alpha\preceq\beta\prec\Omega$ with $\beta\in\mathcal{H}(\emptyset)$ we have
\begin{equation*}
    \mathcal{H}\vdash^\alpha_\rho\Delta, \psi \quad\Rightarrow\quad \mathcal{H}\vdash^\alpha_\rho\Delta, \psi^\beta.  
\end{equation*}
\end{proposition}
\begin{proof}
We use induction over~$\alpha$. First note that we get $k(\psi^\beta)\subseteq k(\psi)\cup\{\beta\}\subseteq\mathcal H(\emptyset)$. Let us now distinguish cases according to the clauses from Definition~\ref{def:H-proofs}. Due to the assumption~$\alpha\prec\Omega$, clause~($\mathsf F^\Omega$) cannot apply. Since $\Delta,\psi$ is an $\mathcal L_1^\Omega$-sequent, this leaves only the clauses ($\bigwedge$) and ($\bigvee$) as well as ($\mathsf{Cut}$). In the case of the latter, we have premises
\begin{equation*}
    \mathcal H\vdash^{\alpha(0)}_\rho\Gamma,\psi,\varphi\quad\text{and}\quad\mathcal H\vdash^{\alpha(1)}_\rho\Gamma,\psi,\neg\varphi
\end{equation*}
with $\alpha(i)\prec\alpha$, where $\varphi$ is an $\mathcal L_\omega^\Omega$-sentence with~$|\varphi|\prec\rho$. The assumption $\rho\preceq\Omega+\omega$ ensures that $\varphi$ is an $\mathcal L_1^\Omega$-sentence, so that the induction hypothesis applies. It gives
\begin{equation*}
    \mathcal H\vdash^{\alpha(0)}_\rho\Gamma,\psi^\beta,\varphi\quad\text{and}\quad\mathcal H\vdash^{\alpha(1)}_\rho\Gamma,\psi^\beta,\neg\varphi,
\end{equation*}
so that we an reapply ($\mathsf{Cut}$) to conclude. When ($\bigwedge$) or ($\bigvee$) is used to infer some formula in~$\Delta$, it is similarly straightforward to reduce to the induction hypothesis.

Now assume that $\psi\simeq\bigvee_{i\prec\iota(\psi)}\psi_i$ was introduced by clause~($\bigvee$) with some premise
\begin{equation*}
    \mathcal H\vdash^{\gamma}_\rho\Delta,\psi,\psi_i,
\end{equation*}
where we have $\gamma\prec\alpha$ as well as $i\prec\min(\iota(\psi),\alpha)$ and $i\in\mathcal H(\emptyset)$. By two applications of the induction hypothesis, we obtain
\begin{equation*}
    \mathcal H\vdash^{\gamma}_\rho\Delta,\psi^\beta,(\psi_i)^\beta.
\end{equation*}
To reapply~($\bigvee$), we check that $\psi^\beta$ is disjunctive with $i\prec\iota(\psi^\beta)$ and $(\psi^\beta)_i\equiv(\psi_i)^\beta$. This is straightforward when~$\psi$ has the form $\psi_0\lor\psi_1$ or $\exists x\,\varphi$ or $I^0_{\theta,\delta}t$ with $\delta\prec\Omega$, where we have $\iota(\psi^\beta)=\iota(\psi)$. Let us now assume that $\psi$ has the form~$I^0_{\theta,\Omega}t$, so that we have~$\psi_j\equiv\theta[t,I^0_{\theta,j}]$ for $j\prec\iota(\psi)=\Omega$. Here $(\psi_j)^\beta$ coincides with $\psi_j$, due to the fact that $\theta$ is an~$\mathcal L_0^X$-formula. Furthermore, we get $\psi^\beta\equiv I^0_{\theta,\beta}t$, which entails that we have $(\psi^\beta)_k\equiv\theta[t,I^0_{\theta,k}]$ for $k\prec\beta=\iota(\psi^\beta)$. Due to the condition~$i\prec\alpha$, we can conclude $i\prec\iota(\psi^\beta)$ as well as $(\psi^\beta)_i\equiv\psi_i\equiv(\psi_i)^\beta$.

Finally, assume that ($\bigwedge$) was used to introduce $\psi\simeq\bigwedge_{i\prec\iota(\psi)}\psi_i$. Similarly to the case of ($\bigvee$), one can reapply ($\bigwedge$) after using the induction hypothesis on both~$\psi$ and the~$\psi_i$. The point is that we have $(\psi^\beta)_i\equiv(\psi_i)^\beta$ for $i\prec\iota(\psi^\beta)\preceq\iota(\psi)$.
\end{proof}

In view of $\neg I^0_{\varphi,\Omega}t\simeq\bigwedge_{\alpha\prec\Omega}\neg\varphi[t,I^0_{\varphi,\alpha}]$, the rule ($\bigwedge$) can have $\Omega$-many premises. If one thinks of~$\Omega$ as the first uncountable ordinal (cf.~the second paragraph before Definition~\ref{def:Gamma-supp}), this means that proof trees can be uncountable. The following definition singles out the formulas that do not have subformulas of the form~$\neg I^0_{\varphi,\Omega}t$. We will see that proofs of these formulas admit a `countable collapse'.

\begin{definition}
    An $\mathcal L_1^\Omega$-formula is a $\Sigma^\Omega$-formula if it is generated as follows:
    \begin{enumerate}[label=(\roman*)]
        \item All literals of~$\mathcal L_0$ (the extension of first-order arithmetic by the predicate~$Q$) are $\Sigma^\Omega$-formulas.
        \item Any formula $I^0_{\varphi,\alpha}t$ with~$\alpha\preceq\Omega$ and any formula~$\neg I^0_{\varphi,\beta}t$ with $\beta\prec\Omega$ (note the strict inequality) is a $\Sigma^\Omega$-formula.
        \item If $\psi_0$ and $\psi_1$ are $\Sigma^\Omega$-formulas, then so are $\psi_0\land\psi_1$ and $\psi_0\lor\psi_1$.
        \item If $\psi$ is a $\Sigma^\Omega$-formula, then so are $\forall x\,\psi$ and $\exists x\,\psi$.
    \end{enumerate}
    By a $\Sigma^\Omega$-sequent we mean a finite set of $\Sigma^\Omega$-sentences. Finally, an $\mathcal L_1^\Omega$-sentence~$\psi$ is a $\Pi^\Omega$-sentence precisely if~$\neg\psi$ is a $\Sigma^\Omega$-sentence.
\end{definition}

Intuitively, $\Pi^\Omega$-formulas do not have subformulas of the form~$I^0_{\varphi,\Omega}t$. We will need the following variant of inversion.

\begin{lemma}[$\mathsf{ATR}_0$]\label{lem:Pi-inversion}
    Consider an $\mathcal L_1^\Omega$-sequent~$\Delta$ as well as a $\Pi^\Omega$-sentence~$\psi$. For any operator~$\mathcal H$ and any $\rho\preceq\Omega+\omega$ and $\beta\prec\Omega$ with $\beta\in\mathcal H(\emptyset)$, we have
    \begin{equation*}
        \mathcal H\vdash^\alpha_\rho\Delta,\psi\quad\Rightarrow\quad\mathcal H\vdash^\alpha_\rho \Delta,\psi^\beta.
    \end{equation*}
\end{lemma}
\begin{proof}
    One argues by induction over~$\alpha$ and case distinction according to the clauses from Definition~\ref{def:H-proofs}. The argument is broadly similar to the proof of Proposition~\ref{prop:boundedness}. In the most interesting case, clause~($\bigwedge$) was used to introduce $\psi\simeq\bigwedge_{i\prec\iota(\psi)}\psi_i$ from premises
    \begin{equation*}
        \mathcal H[i]\vdash^{\alpha(i)}_\rho\Delta,\psi,\psi_i
    \end{equation*}
    with $\alpha(i)\prec\alpha$ for all~$i\prec\iota(\psi)$. We note that $\psi^\beta$ is again a conjunctive $\Pi^\Omega$-sentence with $(\psi^\beta)_i\equiv(\psi_i)^\beta$ for $i\prec\iota(\psi^\beta)\preceq\iota(\psi)$. Two applications of the induction hypothesis yield
    \begin{equation*}
        \mathcal H[i]\vdash^{\alpha(i)}_\rho\Delta,\psi^\beta,\psi_i^\beta
    \end{equation*}
    for $i\prec\iota(\psi^\beta)$. One can reapply ($\bigwedge$) to conclude.

    In the case where ($\bigvee$) was used to introduce $\psi\simeq\bigvee_{i\prec\iota(\psi)}\psi_i$, the crucial point is that we even have~$\iota(\psi^\beta)=\iota(\psi)$, as the $\Pi^\Omega$-sentence~$\psi$ cannot be of the form~$I^0_{\varphi,\Omega} t$. For the same reason, clause~($\mathsf F^\Omega$) cannot apply to~$\psi$ but only to formulas from~$\Delta$.
\end{proof}

In the previous results, we have assumed a cut bound $\Omega+\omega$ to ensure that only~$\mathcal L_1^\Omega$-sentences occur in the proof. The next theorem requires the even sharper bound~$\Omega+1$, which is justified in view of Corollary~\ref{cor:embedding}. The collapsing function
\begin{equation*}
    \psi:\Gamma_{\Omega+Y}\to\psi(\Gamma_{\Omega+Y})=\{\gamma\in\Gamma_{\Omega+Y}:\gamma\prec\Omega\},
\end{equation*}
was introduced in Definition~\ref{def:psi} (where the equality in the codomain holds due to $\Omega=\Gamma_{\Omega+0}$ for the minimal element~$0\in Y=1+X$). By the following theorem, both the proof height and the cut rank are thus collapsed to values below~$\Omega$. In particular, this eliminates applications of clause~($\mathsf F^\Omega$) from Definition~\ref{def:H-proofs} (where the condition $\Omega\preceq\alpha$ was built in to ensure this). While the previous results were true for arbitrary operators, the following theorem requires specific operators~$\mathcal H_\alpha$, which were defined at the end of Section~\ref{sect:ordinal-notations}. 

\begin{theorem}[$\mathsf{ATR}_0$; `Collapsing']\label{thm:collapsing}
    For a $\Sigma^\Omega$-sequent~$\Delta$ and any $\alpha,\beta\in\Gamma_{\Omega+Y}$ with $\alpha\in\mathcal H_\alpha(\emptyset)$, we have
    \begin{equation*}
        \mathcal H_\alpha\vdash^\beta_{\Omega+1}\Delta\quad\Rightarrow\quad\mathcal H_\eta\vdash^{\psi(\eta)}_{\psi(\eta)}\Delta\quad\text{with}\quad\eta=\alpha+\omega^\beta.
    \end{equation*}
\end{theorem}

\begin{proof}
    Concerning the initial condition of Definition~\ref{def:H-proofs}, we get~$\eta\in\mathcal H_\alpha(\emptyset)\subseteq\mathcal H_\eta(\emptyset)$ by Remark~\ref{rmk:H-closure} and Lemma~\ref{lem:operators-monotone}. Proposition~\ref{prop:H-closure-psi} then yields $\psi(\eta)\in\mathcal H_\eta(\emptyset)$. Also, we have $k(\Delta)\subseteq\mathcal H_\alpha(\emptyset)\subseteq\mathcal H_\eta(\emptyset)$.

    We now argue by induction on~$\beta$ and distinguish cases according to the clauses from Definition~\ref{def:H-proofs}. The clauses~($\mathsf{Ax}$) and ($\mathsf F^+$) and~($\mathsf F^-$) cannot occur, due to the fact that $\Delta$ is an~$\mathcal L_1^\Omega$-sequent. Let us first consider an application of clause~($\bigwedge$) to a sentence~$\varphi\simeq\bigwedge_{i\prec\iota(\varphi)}\varphi_i\in\Delta$, where we have premises
    \begin{equation*}
        \mathcal H_\alpha[i]\vdash^{\beta(i)}_{\Omega+1}\Delta,\varphi_i
    \end{equation*}
    with $\beta(i)\prec\beta$ for all~$i\prec\iota(\varphi)$. Since~$\varphi$ is a $\Sigma^\Omega$-sentence, it must be a literal of~$\mathcal L_0$ or of the form $\varphi_0\land\varphi_1$ or $\forall x\,\psi$ or $\neg I^0_{\theta,\gamma}t$ with $\gamma\prec\Omega$ (cf.~Definition~\ref{def:disj}). In the last case we obtain $\iota(\varphi)=\gamma\in k(\Delta)\subseteq\mathcal H_\alpha(\emptyset)$. In the other cases we have $\iota(\varphi)\in\{0,2,\omega\}$, so that we always get
    \begin{equation*}
        \iota(\varphi)\in\mathcal H_\alpha(\emptyset)\cap\psi(\Gamma_{\Omega+Y}).
    \end{equation*}
    For any~$i\prec\iota(\varphi)$, Proposition~\ref{prop:H-closure-psi} now ensures $i\in\mathcal H_\alpha(\emptyset)$ and hence~$\mathcal H_\alpha[i]=\mathcal H_\alpha$. Given that $\varphi$ is a $\Sigma^\Omega$-sentence, we also see that the same holds for~$\varphi_i$. Thus the induction hypothesis yields
    \begin{equation*}
        \mathcal H_{\eta(i)}\vdash^{\psi(\eta(i))}_{\psi(\eta(i))}\Delta,\varphi_i\quad\text{with}\quad\eta(i)=\alpha+\omega^{\beta(i)}.
    \end{equation*}
    Due to~$\eta(i)\in\mathcal H_\alpha(\emptyset)$ we get $\psi(\eta(i))\prec\psi(\eta)$, again by Proposition~\ref{prop:H-closure-psi}. In particular, this allows us to apply weakening (see also Lemma~\ref{lem:operators-monotone}) in order to obtain
    \begin{equation*}
        \mathcal H_\eta[i]\vdash^{\psi(\eta(i))}_{\psi(\eta)}\Delta,\varphi_i\quad\text{for}\quad i\prec\iota(\varphi).
    \end{equation*}
    We can now reapply~($\bigwedge$) to conclude.

    Next, consider the case where ($\bigvee$) is used to infer a sentence $\varphi\simeq\bigvee_{j\prec\iota(\varphi)}\varphi_j\in\Delta$. Here we have a premise $\mathcal H_\alpha\vdash^\gamma_{\Omega+1}\Delta,\varphi_i$ for some $i\prec\min(\iota(\varphi),\alpha)$ with $i\in\mathcal H_\alpha(\emptyset)$ and $\gamma\prec\beta$. The induction hypothesis and weakening yield
    \begin{equation*}
    \mathcal H_\eta\vdash^{\psi(\xi)}_{\psi(\eta)}\Delta,\varphi_i\quad\text{with}\quad\xi=\alpha+\omega^\gamma\prec\alpha+\omega^\beta=\eta,
    \end{equation*}
    where we get $\psi(\xi)\prec\psi(\eta)$ as in the previous case. Crucially, $i\in\mathcal H_\alpha(\emptyset)$ together with $i\prec\iota(\varphi)\preceq\Omega$ entails $i\prec\psi(\alpha+1)\prec\psi(\eta)$ due to Proposition~\ref{prop:H-closure-psi} (as $i\prec\Omega$ amounts to $i\in\psi(\Gamma_{\Omega+X})$ by the paragraph after Assumption~\ref{as:Gamma-Omega-X}). This condition was required in clause~($\bigvee$) of Definition~\ref{def:H-proofs}. We now reapply this clause to conclude.

    In the penultimate case, clause~($\mathsf F^\Omega$) is used to introduce a formula~$I^0_{\varphi,\Omega}t\in\Delta$ from a premise
    \begin{equation*}
        \mathcal H_\alpha\vdash^\gamma_{\Omega+1}\Delta,\varphi[t,I^0_{\varphi,\Omega}]\quad\text{with}\quad\gamma\prec\beta.
    \end{equation*}
    Here $\varphi[t,I^0_{\varphi,\Omega}]$ is a $\Sigma^\Omega$-sentence, because $\varphi$ is a positive~$\mathcal L_0^X$-formula. We can thus apply the induction hypothesis. Together with weakening, it yields
    \begin{equation*}
        \mathcal H_\eta\vdash^{\psi(\xi)}_{\psi(\eta)}\Delta,\varphi[t,I^0_{\varphi,\Omega}]\quad\text{with}\quad\xi=\alpha+\omega^\gamma\prec\alpha+\omega^\beta=\eta,
    \end{equation*}
    where we get $\psi(\xi)\prec\psi(\eta)$ once again. Due to $\xi\in\mathcal H_\alpha(\emptyset)\subseteq\mathcal H_\eta(\emptyset)$, Proposition~\ref{prop:H-closure-psi} also yields $\psi(\xi)\in\mathcal H_\eta(\emptyset)$. We can thus use Proposition~\ref{prop:boundedness} (bounding) to get
    \begin{equation*}
        \mathcal H_\eta\vdash^{\psi(\xi)}_{\psi(\eta)}\Delta,\varphi[t,I^0_{\varphi,\psi(\xi)}].
    \end{equation*}
    In view of $I^0_{\varphi,\Omega}t\simeq\bigvee_{\delta\prec\Omega}\varphi[t,I^0_{\varphi,\delta}]\in\Delta$, we now conclude by an application of~($\bigvee$).

    Finally, we consider an application of ($\mathsf{Cut}$) with premises
    \begin{equation*}
        \mathcal H_\alpha\vdash^\gamma_{\Omega+1}\Delta,\varphi\quad\text{and}\quad\mathcal H_\alpha\vdash^\gamma_{\Omega+1}\Delta,\neg\varphi
    \end{equation*}
    for some~$\gamma\prec\beta$, where~$\varphi$ is an $\mathcal L_\omega^\Omega$-sentence with $|\varphi|\preceq\Omega$. The bound on $|\varphi|$ ensures that $\varphi$ is a $\Sigma^\Omega$-sentence or of the form~$\neg I^0_{\theta,\Omega}$. In the latter case $\neg\varphi\equiv I^0_{\theta,\Omega}t$ is~$\Sigma^\Omega$, so we may invoke duality (recall $\neg\neg\varphi\equiv\varphi$) to assume that~$\varphi$ is a~$\Sigma^\Omega$-sentence. By the induction hypothesis and weakening together with bounding, we get
    \begin{equation*}
        \mathcal H_\eta\vdash^{\psi(\xi)}_{\psi(\eta)}\Delta,\varphi^{\psi(\xi)}\quad\text{with}\quad\xi=\alpha+\omega^\gamma\prec\alpha+\omega^\beta=\eta
    \end{equation*}
    with~$\psi(\xi)\prec\psi(\eta)$ as before. On the other hand, Lemma~\ref{lem:Pi-inversion} yields
    \begin{equation*}
        \mathcal H_\xi\vdash^\gamma_{\Omega+1}\Delta,\neg\varphi^{\psi(\xi)}
    \end{equation*}
    after we have weakened~$\mathcal H_\alpha$ to get~$\psi(\xi)\in\mathcal H_\xi(\emptyset)$ (and where brackets may be omitted in view of~$\neg(\varphi^\delta)\equiv(\neg\varphi)^\delta$). Since $\neg\varphi^{\psi(\xi)}$ is a $\Sigma^\Omega$-sentence, we may now use the induction hypothesis to obtain
    \begin{equation*}
        \mathcal H_\eta\vdash^{\psi(\pi)}_{\psi(\eta)}\Delta,\neg\varphi^{\psi(\xi)}\quad\text{with}\quad\pi=\xi+\omega^\gamma=\alpha+\omega^\gamma\cdot 2\prec\eta,
    \end{equation*}
    where $\pi\in\mathcal H_\alpha(\emptyset)$ entails $\psi(\pi)\prec\psi(\eta)$ by Proposition~\ref{prop:H-closure-psi}. Due to $k(\varphi^{\psi(\xi)})\subseteq\mathcal H_\xi(\emptyset)$ (see the initial condition in Definition~\ref{def:H-proofs}), we have
    \begin{equation*}
        \left|\varphi^{\psi(\xi)}\right|\in\mathcal H_\xi\left(k\left(\varphi^{\psi(\xi)}\right)\right)=\mathcal H_\xi(\emptyset)
    \end{equation*}
    as well as $|\varphi^{\psi(\xi)}|\prec\Omega$. By Proposition~\ref{prop:H-closure-psi} we infer $|\varphi^{\psi(\xi)}|\prec\psi(\xi+1)\prec\psi(\eta)$. We can now apply~($\mathsf{Cut}$) to obtain $\mathcal H_\eta\vdash^{\psi(\eta)}_{\psi(\eta)}\Delta$, as desired.
\end{proof}

After collapsing, not only the proof heigth but also the cut rank is below~$\Omega$. This removes the obstacle to cut elimination that occurs at rank~$\Omega+1$ (see Lemma~\ref{lem:basic-cut-elim}). We can thus employ predicative cut elimination once again.

\begin{proposition}[$\mathsf{ATR}_0$]
    For~$\alpha,\beta\prec\Omega$ with $\beta\in\mathcal H(\emptyset)$ and any $\mathcal L_1^\Omega$-sequent~$\Delta$, we have
    \begin{equation*}
        \mathcal H\vdash^\gamma_{\alpha+\omega^\beta}\Delta\quad\Rightarrow\quad\mathcal H\vdash^{\varphi_\beta(\gamma)}_\alpha\Delta.
    \end{equation*}
\end{proposition}
\begin{proof}
    One argues just like in the proof of Proposition~\ref{prop:pred-cut-elim}.
\end{proof}

As a typical application of ordinal analysis, we now derive a consistency result. It concerns the proof calculus from Definition~\ref{def:inf-proofs}. We note that the well-order~$X$ has been fixed in Assumption~\ref{as:Gamma-Omega-X}. Also recall that the empty sequent is a canonical representation of contradiction, since a sequent is interpreted as the disjunction over its elements.

\begin{corollary}[$\mathsf{ATR}_0$]
    We do not have $\vdash^x\emptyset$ for any~$x\in X$.
\end{corollary}
\begin{proof}
Towards a contradiction, assume that we have $\vdash^x\emptyset$. By Corollary~\ref{cor:embedding} (embedding), we get
\begin{equation*}
    \mathcal H_0\vdash^{\eta}_{\Omega+1}\emptyset\quad\text{with}\quad\eta=\Gamma_{\Omega+1+x}.
\end{equation*}
Now Theorem~\ref{thm:collapsing} (collapsing) yields $\mathcal H_\eta\vdash^{\psi(\eta)}_{\psi(\eta)}\emptyset$. Due to the previous proposition (partial cut elimination), we can infer
\begin{equation*}
\mathcal H_\eta\vdash^\xi_0\emptyset\quad\text{with}\quad\xi=\varphi_{\psi(\eta)}(\psi(\eta)).
\end{equation*}
But this must be false, since no clause from Definition~\ref{def:H-proofs} can apply. Specifically, ($\mathsf{Cut}$) is excluded because $|\varphi|\prec 0$ never holds, while all other clauses would require that~$\emptyset$ contains some sentence.
\end{proof}

Finally, we can derive the main theorem of this paper, which was stated in the introduction (and which we recall in the following proof).

\begin{proof}[Proof of Theorem~\ref{thm:main}]
Working in $\mathsf{ACA}_0$, we have to show that $\Pi^1_1\text{-}\mathsf{CA}^\Gamma_0$ (recalled below) is equivalent to the well-ordering principle that $\psi(\Gamma_{\Omega+Y})$ is well-founded for every well-order~$Y$. The forward implication has been established in Theorem~\ref{thm:WO_principle}.

For the backward implication, we first boot up to~$\mathsf{ATR}_0$, which allows us to use all the previous results: By Proposition~\ref{prop:wo-princ-robust}, our well-ordering principle entails that $\Gamma_Y\subseteq\Gamma_{\Omega+Y}$ is well-founded for any well-order~$Y$. As shown by Rathjen~\cite{rathjen-atr}, it follows that each subset of~$\mathbb N$ is contained in a (countable coded) $\omega$-model of~$\mathsf{ATR}_0$. Given that the axioms of~$\mathsf{ATR}_0$ are $\Pi^1_2$, this entails that they are true.

The principle $\Pi^1_1\text{-}\mathsf{CA}^\Gamma_0$ asserts that each set $Z\subseteq\mathbb N$ lies in some $\omega$-model that satis\-fies $\mathsf{ATR}_0+\Pi^1_1\textsf{-CA}(Z)$. To show that this follows from our well-ordering principle, fix~$\mathcal Q=Z$ as the interpretation of the predicate symbol~$Q$ from the language $\mathcal L_\omega$ (see the paragraph before Definition~\ref{def:inf-proofs}). We will show that $\nvdash^x\emptyset$ holds for every well-order~$X$ and all~$x\in X$. By Corollary~\ref{cor:omega-completeness}, this yields an $\omega$-model of~$\mathsf{ID}^\Gamma$ in which $Q$ is interpreted by~$Z$. Due to Proposition~\ref{prop:ID-to-Pi11}, we obtain an $\omega$-model that satisfies $\mathsf{ATR}_0+\Pi^1_1\textsf{-CA}(Z)$, as desired.

To establish the open claim, consider an arbitrary well-order~$X$. We extend it into~$Y=1+X$ by adding a new minimal element. In view of Proposition~\ref{prop:wo-princ-robust}, our well-ordering principle entails that $\Gamma_{\Omega+Y}$ is well-founded. Thus assumption~\ref{as:Gamma-Omega-X} is satisfied and hereby discharged. The last corollary yields $\nvdash^x\emptyset$ for all~$x\in X$, which proves the claim from the previous paragraph.
\end{proof}

\bibliographystyle{amsplain}
\bibliography{Fraisse-PartialPred}

\end{document}